\RequirePackage[l2tabu,orthodox]{nag}		
\documentclass[reqno]{amsart}
\usepackage[margin=1.6in,bottom=1.25in]{geometry}		

\usepackage{amsmath}		
\usepackage{amssymb}		
\usepackage{amsfonts}		
\usepackage{amsthm}		
\usepackage[foot]{amsaddr}		

\usepackage{mathtools}		

\mathtoolsset{%
}

\usepackage[utf8]{inputenc}		
\usepackage[T1]{fontenc}		

\usepackage[
cal=cm,
]
{mathalfa}

\usepackage{dsfont}		

\usepackage{inconsolata}

\usepackage[proportional,tabular,lining,sf,mono=false]{libertine}

\usepackage{acronym}		
\renewcommand*{\aclabelfont}[1]{\acsfont{#1}}		
\newcommand{\acli}[1]{\textit{\acl{#1}}}		
\newcommand{\acdef}[1]{\textit{\acl{#1}} \textup{(\acs{#1})}\acused{#1}}		

\usepackage[labelfont={bf,small},labelsep=colon,font=small]{caption}	
\usepackage[dvipsnames,svgnames]{xcolor}		
\colorlet{MyRed}{Crimson!75!Black}
\colorlet{MyBlue}{MediumBlue}
\colorlet{MyGreen}{DarkGreen!80!Black}

\usepackage{setspace}
\usepackage[]{titlesec}		
\titleformat{name=\section}{}{\thetitle.}{0.8em}{\centering\scshape}
\titleformat{name=\subsection}[runin]{}{\thetitle.}{0.5em}{\bfseries}[.]
\titleformat{name=\subsubsection}[runin]{}{\thetitle.}{0.5em}{\bfseries}[.]

\titleformat{name=\paragraph,numberless}[runin]{}{}{0em}{\bfseries}[]
\titlespacing{\paragraph}{0em}{\medskipamount}{1em}
\titleformat{name=\subparagraph,numberless}[runin]{}{}{0em}{}[.]
\titlespacing{\subparagraph}{0em}{0em}{0.5em}

\newcommand{\afterhead}{.}
\newcommand{\para}[1]{\paragraph{\textbf{#1\afterhead}}}

\usepackage{pifont}
\newcommand{\cmark}{\checkmark}		
\newcommand{\xmark}{\ding{53}}

\usepackage{subcaption}		
\usepackage{tikz}		
\usetikzlibrary{calc,patterns}

\usepackage{array}		
\usepackage{booktabs}		
\usepackage[inline,shortlabels]{enumitem}		
\setenumerate{itemsep=\smallskipamount,topsep=\smallskipamount,left=1ex}

\usepackage[kerning=true]{microtype}		

\usepackage{latexsym}		
\usepackage{xspace}		

\usepackage[numbers,compress]{natbib}		

\newcommand{\citef}[1]{\citeauthor{#1} \citep{#1}}

\usepackage{hyperref}
\hypersetup{
colorlinks=true,
linktocpage=true,
pdfstartview=FitH,
breaklinks=true,
pdfpagemode=UseNone,
pageanchor=true,
pdfpagemode=UseOutlines,
plainpages=false,
bookmarksnumbered,
bookmarksopen=false,
bookmarksopenlevel=1,
hypertexnames=true,
pdfhighlight=/O,
urlcolor=MyBlue,linkcolor=MyBlue,citecolor=MyBlue,	
pdftitle={},
pdfauthor={},
pdfsubject={},
pdfkeywords={},
pdfcreator={pdfLaTeX},
pdfproducer={LaTeX with hyperref}
}

\usepackage[sort&compress,capitalize,nameinlink]{cleveref}		
\crefname{algorithm}{Algorithm}{Algorithms}

\usepackage{algorithm}		
\usepackage{algpseudocode}		

\usepackage{thmtools}		
\usepackage{thm-restate}		

\theoremstyle{plain}
\newtheorem{theorem}{Theorem}		
\newtheorem{lemma}{Lemma}		
\newtheorem{proposition}{Proposition}		

\newtheorem*{corollary*}{Corollary}		

\theoremstyle{definition}
\newtheorem{definition}{Definition}		
\newtheorem{assumption}{Assumption}		
\newtheorem{example}{{\raisebox{1pt}{\small$\blacktriangleright$}} Example}		

\newtheorem*{definition*}{Definition}		
\newtheorem*{assumption*}{Assumptions}		
\newtheorem*{example*}{Example}		

\theoremstyle{remark}
\newtheorem*{remark*}{Remark}		

\newcounter{proofpart}

\numberwithin{remark}{section}		

\usepackage[showdeletions]{color-edits}		
\usepackage[normalem]{ulem}		
\newcommand{\debug}[1]{#1}		

\newcommand{\newmacro}[2]{\newcommand{#1}{\debug{#2}}}		

\DeclarePairedDelimiter{\bracks}{[}{]}		
\DeclarePairedDelimiter{\parens}{(}{)}		

\DeclarePairedDelimiter{\abs}{\lvert}{\rvert}		

\DeclarePairedDelimiterX{\setdef}[2]{\{}{\}}{#1:#2}		
\DeclarePairedDelimiterXPP{\exclude}[1]{\mathopen{}\setminus}{\{}{\}}{}{#1}

\newcommand{\cf}{cf.\xspace}		
\newcommand{\eg}{e.g.,\xspace}		
\newcommand{\ie}{i.e.,\xspace}		

\newcommand{\textpar}[1]{\textup(#1\textup)}		

\newcommand{\txs}{\textstyle}		

\newcommand{\alt}[1]{#1'}		

\newcommand{\N}{\mathbb{N}}		
\newcommand{\R}{\mathbb{R}}		

\DeclareMathOperator{\bigoh}{\mathcal O}		

\newmacro{\coef}{\lambda}		
\newmacro{\dd}{\:d}		
\newmacro{\subs}{\leftarrow}      

\newcommand{\eps}{\varepsilon}		

\DeclareMathOperator{\dom}{dom}		

\newcommand{\from}{\colon}		

\newmacro{\points}{\mathcal{X}}		
\newmacro{\intpoints}{\points^{\circ}}		
\newmacro{\point}{x}		
\newmacro{\pointalt}{\alt\point}		

\newmacro{\dpoints}{\mathcal{Y}}		
\newmacro{\dpoint}{y}		
\newmacro{\dpointalt}{\alt\dpoint}		

\newmacro{\base}{p}		
\newmacro{\basealt}{q}		

\DeclareMathOperator{\dist}{dist}		

\newmacro{\open}{\mathcal{U}}		
\newmacro{\cpt}{\mathcal{K}}		
\newmacro{\nhd}{\mathcal{U}}		

\newmacro{\start}{1}		
\newmacro{\afterstart}{2}		
\newmacro{\prestart}{0}		
\newmacro{\running}{\start,\afterstart,\dotsc}		

\newmacro{\run}{t}		
\newmacro{\runalt}{s}		
\newmacro{\nRuns}{T}		
\newmacro{\runs}{\mathcal{\nRuns}}		

\newmacro{\state}{X}		
\newmacro{\dstate}{Y}		
\newmacro{\aux}{Z}		

\newcommand{\new}[1][\point]{#1^{+}}		

\newcommand{\init}[1][\state]{\debug{#1}_{\start}}		
\newcommand{\iter}[1][\state]{\debug{#1}_{\runalt}}		
\newcommand{\nextiter}[1][\state]{\debug{#1}_{\runalt+1}}		
\newcommand{\prev}[1][\state]{\debug{#1}_{\run-1}}		
\newcommand{\curr}[1][\state]{\debug{#1}_{\run}}		
\renewcommand{\next}[1][\state]{\debug{#1}_{\run+1}}		
\newcommand{\last}[1][\state]{\debug{#1}_{\nRuns}}		

\newcommand{\avg}[1][\state]{\debug{\bar#1}_{\nRuns}}		

\newmacro{\ctime}{t}		
\newmacro{\ctimealt}{s}		
\newmacro{\cstart}{0}		

\newmacro{\horizon}{T}		

\newmacro{\vecspace}{\mathcal{V}}		
\newmacro{\vdim}{d}		
\newmacro{\vvec}{x}		
\newmacro{\bvec}{e}		
\newmacro{\unitvec}{z}		

\newmacro{\subspace}{\mathcal{W}}		
\newmacro{\thull}{\mathcal{Z}}		
\newmacro{\tvec}{z}		

\DeclarePairedDelimiterX{\braket}[2]{\langle}{\rangle}{#1,#2}		
\newcommand{\dual}[1]{#1^{\ast}}		

\newmacro{\dspace}{\dual\vecspace}		
\newmacro{\dvec}{v}		
\newmacro{\dbvec}{\eps}		

\newmacro{\ones}{\mathbf{1}}		
\newmacro{\mat}{\mathbf{M}}		
\newmacro{\eye}{\mathbf{I}}		
\newmacro{\pointmat}{\mathbf{X}}		
\newmacro{\tmat}{\mathbf{Z}}		
\newmacro{\dmat}{\mathbf{V}}		
\newmacro{\hmat}{\mathbf{H}}		

\DeclareMathOperator{\relint}{ri}		

\newmacro{\cvx}{\mathcal{C}}		
\newmacro{\subd}{\partial}		

\newmacro{\strong}{\alpha}		
\newmacro{\smooth}{L}		
\newmacro{\lips}{L}		

\newmacro{\hreg}{h}		
\newmacro{\breg}{D}		
\newmacro{\proxmap}{P}		
\newmacro{\mirror}{Q}		
\newmacro{\fench}{F}		
\newmacro{\hstr}{K}		
\newmacro{\depth}{H}		
\newmacro{\iBreg}{\breg_{\start}}		
\newmacro{\zone}{\mathbb{D}}		
\newmacro{\subpoints}{\points_{\circ}}		

\DeclarePairedDelimiterXPP{\prox}[2]{\proxmap_{#1}}{(}{)}{}{#2}		

\DeclareMathOperator*{\argmax}{arg\,max}		
\DeclareMathOperator*{\argmin}{arg\,min}		

\newmacro{\obj}{f}		
\newmacro{\objalt}{g}		
\newmacro{\sobj}{F}		

\newmacro{\param}{\theta}		
\newmacro{\params}{\Theta}		

\newmacro{\gvec}{g}		
\newmacro{\vecfield}{V}		
\newmacro{\oper}{A}		

\newmacro{\gbound}{G}		
\newmacro{\vbound}{G}		

\newcommand{\sol}[1][\point]{#1^{\ast}}		
\newcommand{\sols}{\sol[\points]}		

\newmacro{\minmax}{\Phi}		

\newmacro{\minvar}{\point_{1}}		
\newmacro{\minvaralt}{\alt\minvar}		
\newmacro{\minvars}{\points_{1}}		

\newmacro{\maxvar}{\point_{2}}		
\newmacro{\maxvars}{\points_{2}}		
\newmacro{\maxvaralt}{\alt\maxvar}		

\newcommand{\eq}{\sol}		

\newmacro{\play}{i}		
\newmacro{\playalt}{j}		
\newmacro{\nPlayers}{n}		
\newmacro{\players}{\mathcal{N}}		

\newmacro{\pure}{k}		
\newmacro{\purealt}{l}		
\newmacro{\nPures}{m}		
\newmacro{\pures}{\mathcal{M}}		

\newmacro{\cost}{c}		
\newmacro{\loss}{\ell}		
\newmacro{\pay}{u}		
\newmacro{\payv}{v}		
\newmacro{\pot}{\obj}		

\newmacro{\game}{\mathcal{G}}		
\newmacro{\gamefull}{\game(\players,\points,\pay)}		

\newmacro{\fingame}{\Gamma}		
\newmacro{\fingamefull}{\Gamma(\players,\pures,\pay)}		

\DeclareMathOperator{\ex}{\mathbb{E}}		
\DeclareMathOperator{\prob}{\mathbb{P}}		

\newmacro{\sample}{\omega}		
\newmacro{\samples}{\Omega}		
\newmacro{\filter}{\mathcal{F}}		
\newmacro{\probspace}{(\samples,\filter,\prob)}		

\newmacro{\event}{E}       
\newmacro{\eventalt}{H}       

\newmacro{\mean}{\mu}		
\newmacro{\sdev}{\sigma}		
\newmacro{\variance}{\sdev^{2}}		

\newmacro{\dkl}{D_{\mathrm{KL}}}		
\newcommand{\as}{\debug{\textpar{a.s.}}\xspace}		

\providecommand\given{}		

\DeclarePairedDelimiterXPP{\exof}[1]{\ex}{[}{]}{}{
\renewcommand\given{\nonscript\:\delimsize\vert\nonscript\:\mathopen{}} #1}

\DeclarePairedDelimiterXPP{\probof}[1]{\prob}{(}{)}{}{
\renewcommand\given{\nonscript\:\delimsize\vert\nonscript\:\mathopen{}} #1}

\newmacro{\step}{\gamma}		
\newmacro{\temp}{\eta}		

\newmacro{\signal}{g}		
\newmacro{\error}{Z}		
\newmacro{\noise}{U}		
\newmacro{\bias}{b}		
\newmacro{\brown}{W}		

\newmacro{\snoise}{\xi}		
\newmacro{\sbias}{\psi}		

\newmacro{\sbound}{M}		
\newmacro{\bbound}{B}		
\newmacro{\noisedev}{\sdev}		
\newmacro{\noisevar}{\variance}		

\newmacro{\mix}{\delta}		
\newmacro{\pexp}{p}		
\newmacro{\qexp}{q}		

\newmacro{\pert}{w}		
\newmacro{\pertvar}{W}		
\newmacro{\unitvar}{\unitvec}		

\newmacro{\flow}{\Phi}		

\newmacro{\graph}{\mathcal{G}}
\newmacro{\vertices}{\mathcal{V}}
\newmacro{\edges}{\mathcal{E}}

\DeclarePairedDelimiterX{\product}[2]{\langle}{\rangle}{#1,#2}		
\DeclarePairedDelimiter{\norm}{\lVert}{\rVert}		
\DeclarePairedDelimiterXPP{\dnorm}[1]{}{\lVert}{\rVert}{_{\ast}}{#1}		

\DeclarePairedDelimiterXPP{\twonorm}[1]{}{\lVert}{\rVert}{_{2}}{#1}		
\DeclarePairedDelimiterXPP{\pnorm}[1]{}{\lVert}{\rVert}{_{2}}{#1}		

\newmacro{\gmat}{g}		
\newmacro{\gdist}{\dist_{\gmat}}
\newmacro{\ball}{\mathbb{B}}		
\newmacro{\sphere}{\mathbb{S}}		

\newmacro{\const}{c}
\newcommand{\method}{\debug{\textsc{AdaMir}}\xspace}
\newcommand{\adagrad}{\debug{\textsc{AdaGrad}}\xspace}
\newcommand{\accelegrad}{\debug{\textsc{AcceleGrad}}\xspace}
\newcommand{\unixgrad}{\debug{\textsc{UnixGrad}}\xspace}
\newcommand{\adaprox}{\debug{\textsc{AdaProx}}\xspace}

\newmacro{\fins}{F}

\newmacro{\res}{\delta}		
\newmacro{\cumres}{S}		

\newmacro{\commbound}{H}		
\newmacro{\smoothbound}{A}		
\newmacro{\stochbound}{B}		

\newmacro{\price}{p}
\newmacro{\weight}{w}

\begin{document}


\newcommand{\longtitle}{\uppercase{Adaptive first-order methods revisited:\\
Convex optimization without Lipschitz requirements}}
\newcommand{\runtitle}{\uppercase{Adaptive First-Order Methods without Lipschitz Requirements}}

\title
[\runtitle]		
{\longtitle}		

\author[K.~Antonakopoulos]{Kimon Antonakopoulos$^{\ast}$}
\email{kimon.antonakopoulos@inria.fr}
\author
[P.~Mertikopoulos]
{Panayotis Mertikopoulos$^{\ast,\diamond}$}
\address{$^{\ast}$\,%
Univ. Grenoble Alpes, CNRS, Inria, Grenoble INP, LIG, 38000, Grenoble, France.}
\address{$^{\diamond}$\,%
Criteo AI Lab.}
\email{panayotis.mertikopoulos@imag.fr}

\newacro{LHS}{left-hand side}
\newacro{RHS}{right-hand side}
\newacro{iid}[i.i.d.]{independent and identically distributed}
\newacro{lsc}[l.s.c.]{lower semi-continuous}
\newacro{NE}{Nash equilibrium}
\newacroplural{NE}[NE]{Nash equilibria}

\newacro{LC}{Lipschitz continuity}
\newacro{LS}{Lipschitz smoothness}
\newacro{RC}{relative continuity}
\newacro{RS}{relative smoothness}

\newacro{SFO}{stochastic first-order oracle}
\newacro{DGF}{distance-generating function}

\newacro{PR}{proportional response}
\newacro{EGD}{entropic gradient descent}
\newacro{VI}{variational inequality}
\newacroplural{VI}[VIs]{variational inequalities}
\newacro{AMD}[\method]{adaptive mirror descent}
\newacro{UPGD}{universal primal gradient descent}
\newacro{GMP}{generalized mirror-prox}
\newacro{MD}{mirror descent}
\newacro{NL}[\textsc{NoLips}]{non-Lipschitz}

\begin{abstract}
%
%
We propose a new family of adaptive first-order methods for a class of convex minimization problems that may fail to be Lipschitz continuous or smooth in the standard sense.
Specifically, motivated by a recent flurry of activity on \acf{NL} optimization, we consider problems that are continuous or smooth relative to a reference Bregman function \textendash\ as opposed to a global, ambient norm (Euclidean or otherwise).
These conditions encompass a wide range of problems with singular objective,
such as Fisher markets, Poisson tomography, $D$-design, and the like.
In this setting, the application of existing order-optimal adaptive methods \textendash\ like \unixgrad or \accelegrad \textendash\ is not possible, especially in the presence of randomness and uncertainty.
The proposed method, \acdef{AMD}, aims to close this gap by concurrently achieving min-max optimal rates in problems that are relatively continuous or smooth, including stochastic ones.
\end{abstract}



\allowdisplaybreaks		
\acresetall		
\maketitle

\section{Introduction}
\label{sec:introduction}

Owing to their wide applicability and flexibility,
first-order methods continue to occupy the forefront of research in learning theory and continuous optimization.
Their analysis typically revolves around two basic regularity conditions for the problem at hand:
\begin{enumerate*}
[(\itshape i\hspace*{.5pt}\upshape)]
\item
Lipschitz continuity of the problem's objective,
and\,/\,or
\item
Lipschitz continuity of its gradient (also referred to as \emph{Lipschitz smoothness}).
\end{enumerate*}
Depending on which of these conditions holds, the lower bounds for first-order methods with perfect gradient input are $\Theta(1/\sqrt{\nRuns})$ and $\Theta(1/\nRuns^{2})$ after $\nRuns$ gradient queries, and they are achieved by gradient descent and Nesterov's fast gradient algorithm respectively \citep{Nes83,Nes04,Xia10}.
By contrast, if the optimizer only has access to stochastic gradients (as is often the case in machine learning and distributed control), the corresponding lower bound is $\Theta(1/\sqrt{\nRuns})$ for both classes \citep{Nes04,NY83,Bub15}.

This disparity in convergence rates has led to a surge of interest in adaptive methods that can seamlessly interpolate between these different regimes.
Two state-of-the-art methods of this type are the \accelegrad and \unixgrad algorithms of \citet{LYC18} and \citet{KLBC19}:
both algorithms simultaneously achieve
an $\bigoh(1/\sqrt{\nRuns})$ value convergence rate in non-smooth problems,
an $\bigoh(1/\nRuns^{2})$ rate in smooth problems,
and
an $\bigoh(1/\sqrt{\nRuns})$ average rate if run with bounded, unbiased stochastic gradients (the smoothness does not affect the rate in this case).
In this regard, \unixgrad and \accelegrad both achieve a ``best of all worlds'' guarantee which makes them ideal as off-the-shelf solutions for applications where the problem class is not known in advance \textendash\ \eg as in online traffic routing, game theory, etc.

At the same time, there have been considerable efforts in the literature to account for problems that do not adhere to these Lipschitz regularity requirements \textendash\ such as Fisher markets, quantum tomography, $D$-design, Poisson deconvolution / inverse problems, and many other examples \citep{BDX11,BBT17,LFN18,Lu19,Teb18,BSTV18}.
The reason that the Lipschitz framework fails in this case is that, even when the problem's domain is bounded, the objective function exhibits singularities at the boundary, so it cannot be Lipschitz continuous or smooth.
As a result, no matter how small we pick the step-size of a standard gradient method (adaptive or otherwise), the existence of domains with singular gradients can \textendash\ and typically \emph{does} \textendash\ lead to catastrophic oscillations (especially in the stochastic case).

A first breakthrough in this area was provided by \citet{BDX11} and, independently, \citet{BBT17} and \citet{LFN18}, who considered a ``Lipschitz-like'' gradient continuity condition for problems with singularities.%
\footnote{This condition was first examined by \citet{BDX11} in the context of Fisher markets.
The analysis of \citet{BBT17} and \citet{LFN18} is much more general, but several ideas are already present in \cite{BDX11}.}
At around the same time, \citet{Lu19} and \citet{Teb18} introduced a ``relative continuity'' condition which plays the same role for Lipschitz continuity.
Collectively, instead of using a global norm as a metric, these conditions employ a Bregman divergence as a measure of distance, and they replace gradient descent with \acli{MD} \citep{NY83,Bub15}.

In these extended problem classes, \emph{non-adaptive} \acl{MD} methods achieve
an $\bigoh(1/\sqrt{\nRuns})$ value convergence rate in relatively continuous problems \citep{Lu19,ABM20},
an $\bigoh(1/\nRuns)$ rate in relatively smooth problems \citep{BDX11,BBT17},
and
an $\bigoh(1/\sqrt{\nRuns})$ average rate if run with stochastic gradients in relatively continuous problems \citep{Lu19,ABM20}.
Importantly, the $\bigoh(1/\nRuns)$ rate for relatively smooth problems \emph{does not match} the $\bigoh(1/\nRuns^{2})$ rate for standard Lipschitz smooth problems:
in fact, even though \cite{HRX18} proposed a tentative path towards faster convergence in certain non-Lipschitz problems, \citet{DTAB19} recently established an $\Omega(1/\nRuns)$ lower bound for problems that are relatively-but-not-Lipschitz smooth.

\para{Our contributions}

Our aim in this paper is to provide an \emph{adaptive, parameter-agnostic} method that simultaneously achieves order-optimal rates in the above ``\acl{NL}'' framework.
By design, the proposed method \textendash\ which we call \acdef{AMD} \textendash\ has the following desirable properties:
\begin{enumerate}
[itemsep=0pt]
\item
When run with perfect gradients, the trajectory of queried points converges, and the method's rate of convergence in terms of function values interpolates between $\bigoh(1/\sqrt{\nRuns})$ and $\bigoh(1/\nRuns)$ for relatively continuous and relatively smooth problems respectively.
\item
When run with stochastic gradients, the method attains an $\bigoh(1/\sqrt{\nRuns})$ average rate of convergence.
\item
The method applies to both constrained and unconstrained problems, without requiring
a finite Bregman diameter
or
knowledge of a compact domain containing a solution.
\end{enumerate}

The enabling apparatus for these properties is an adaptive step-size policy in the spirit of \cite{DHS11,MS10,BL19}.
However, a major difference \textendash\ and technical difficulty \textendash\ is that the relevant definitions cannot be stated in terms of global norms, because the variation of non-Lipschitz function explodes at the boundary of the problem's domain (put differently, gradients may be unbounded even over bounded domains).
For this reason, our policy relies on the aggregation of a suitable sequence of ``Bregman residuals'' that stabilizes seamlessly when approaching a smooth solution, thus enabling the method to achieve faster convergence rates.

\para{Related work}

Beyond the references cited above, problems with singular objectives were treated in a recent series of papers in the context of online and stochastic optimization \cite{ABM20,ZPSH20};
however, the proposed methods are \emph{not} adaptive, and \emph{they do not interpolate} between different problem classes.


\begin{table}[tbp]
\centering
\footnotesize
\renewcommand{\arraystretch}{1.3}

\begin{tabular}{lcccccc}
\toprule
	&Constr. $/$ Uncon.
	&Stoch. (L)
	&Order-optimal
	&\acs{RC}
	&\acs{RS}
	&Stoch. (R)
	\\
\midrule
\adagrad\hfill\cite{DHS11}
	&\cmark $/$\cmark
	&\cmark
	&\xmark
	&\xmark
	&\xmark
	&\xmark
	\\
\accelegrad\;\;\hfill\cite{LYC18}
	&\xmark $/$ \cmark
	&\cmark
	&\cmark
	&\xmark
	&\xmark
	&\xmark
	\\
\unixgrad\hfill\cite{KLBC19}
	&\cmark $/$ \xmark
	&\cmark
	&\cmark
	&\xmark
	&\xmark
	&\xmark
	\\
\acs{UPGD}\hfill\cite{Nes15}
	&\cmark $/$ \cmark
	&\xmark
	&\cmark
	&\xmark
	&\xmark
	&\xmark
	\\
\acs{GMP}\hfill\cite{SGTP+19}
	&\cmark $/$ \cmark
	&\xmark
	&partial
	&\xmark
	&$1/\nRuns$
	&\xmark
	\\
\adaprox\hfill\cite{ABM21}
	&\cmark $/$ \cmark
	&\xmark
	&\cmark
	&partial
	&partial
	&\xmark
	\\
\method\hfill[ours]
	&\cmark $/$ \cmark
	&\cmark
	&\cmark
	&$1/\sqrt{\nRuns}$
	&$1/\nRuns$
	&$1/\sqrt{\nRuns}$
	\\
\bottomrule
\end{tabular}
\medskip
\caption{%
Overview of related work.
For the purposes of this table, (L) refers to ``Lipschitz'' and (R) to ``relative'' continuity or smoothness respectively.
``Order-optimal'' means that the algorithm attains the best rates for the worst instance in the class it was designed for (see cited papers for the details).
Logarithmic factors are ignored throughout;
we also note that the $\bigoh(1/\nRuns)$ rate in the RS column is, in general, unimprovable \cite{DTAB19}.}
\label{tab:related}
\end{table}


In the context of adaptive methods, the widely used \adagrad algorithm of \citef{DHS11} and \citef{MS10} was recently shown to interpolate between an $\bigoh(1/\sqrt{\nRuns})$ and $\bigoh(1/\nRuns)$ rate of convergence \citep{LYC18,LO19}.
More precisely, \citet{LO19} showed that a specific, ``one-lag-behind'' variant of \adagrad with prior knowledge of the smoothness modulus achieves an $\bigoh(1/\nRuns)$ rate in smooth, unconstrained problems;
concurrently, \citet{LYC18} obtained the same rate in a parameter-agnostic context.
In both cases, \adagrad achieves an $\bigoh(1/\sqrt{\nRuns})$ rate of convergence in stochastic problems (though with somewhat different assumptions for the randomness).

In terms of rate optimality for Lipschitz smooth problems, \adagrad is outperformed by \accelegrad \citep{LYC18} and \unixgrad \citep{KLBC19}:
these methods both achieve an $\bigoh(1/\nRuns^{2})$ rate in Lipschitz smooth problems, and an $\bigoh(1/\sqrt{\nRuns})$ in stochastic problems with bounded gradient noise.
By employing an efficient line-search step, the \acdef{UPGD} algorithm of \citef{Nes15} achieves order-optimal guarantees in the wider class of Hölder continuous problems (which includes the Lipschitz continuous and smooth cases as extreme cases);
however, \ac{UPGD} does not cover stochastic problems or problems with relatively continuous\,/\,smooth objectives.

As far as we are aware, the closest works to our own are
the \acdef{GMP} algorithm of \cite{SGTP+19}
and
the \adaprox method of \cite{ABM20},
both designed for variational inequality problems.
The \ac{GMP} algorithm can achieve interpolation between different classes of Hölder continuous problems and can adapt to the problem's relative smoothness modulus, but it does not otherwise interpolate between the relatively smooth and relatively continuous classes.
Moreover, \ac{GMP} requires knowledge of a ``domain of interest'' containing a solution of the problem;
in this regard, it is similar to \accelegrad \cite{LYC18} (though it does not require an extra projection step).
The recently proposed \adaprox method of \citef{ABM21} also achieves a similar interpolation result under a set of assumptions that are closely related \textendash\ but not equivalent \textendash\ to the relatively continuous/smooth setting of our paper.
In any case, neither of these papers covers the stochastic case;
to the best of our knowledge, \method is the first method in the literature capable of adaptiving to relatively continuous\,/\,smooth objectives, even under uncertainty.
For convenience, we detail these related works in \cref{tab:related} above.

\section{Problem setup and preliminaries}
\label{sec:setup}

\para{Problem statement}

Throughout the sequel, we will focus on convex minimization problems of the general form
\begin{equation}
\label{eq:opt}
\tag{Opt}
\begin{aligned}
\textrm{minimize}
	&\quad
	\obj(\point),
	\\
\textrm{subject to}
	&\quad
	\point \in \points.
\end{aligned}
\end{equation}
In the above,
$\points$ is a convex subset of a normed $\vdim$-dimensional space $\vecspace \cong \R^{\vdim}$,
and
$\obj\from\vecspace\to\R\cup\{\infty\}$ is a proper \ac{lsc} convex function with $\dom\obj = \setdef{\point\in\vecspace}{\obj(\point)<\infty} = \points$.
Compared to standard formulations, we stress that
\emph{$\points$ is not assumed to be compact, bounded, or even closed}.
This lack of closedness will be an important feature for our analysis because we are interested in objectives that may develop singularities at the boundary of their domain;
for a class of relevant examples of this type, see \cite{BDX11,BBT17,LFN18,BBDV09,BSTV18,ABM20,ZPSH20} and references therein.


To formalize our assumptions for \eqref{eq:opt}, we will write
\(
\subd\obj(\point)
\)
for the \emph{subdifferential} of $\obj$ at $\point$, and
\(
\subpoints
	\equiv \dom\subd\obj
	= \setdef{\point\in\points}{\subd\obj(\point) \neq \varnothing}
\)
for the \emph{domain of subdifferentiability} of $\obj$.
Formally, elements of $\subd\obj$ will be called subgradients, and we will treat them throughout as elements of the dual space $\dspace$ of $\vecspace$.
By standard results, we have $\relint\points \subseteq \subpoints \subseteq \points$, and any solution $\sol$ of \eqref{eq:opt} belongs to $\subpoints$ \cite[Chap.~26]{Roc70};
to avoid trivialities, we will make the following blanket assumption.

\begin{assumption}
\label{asm:sol}
The solution set $\sols \equiv \argmin\obj \subseteq \subpoints$ of \eqref{eq:opt} is nonempty.
\end{assumption}


Two further assumptions that are standard in the literature (but which we relax in the sequel) are:
\begin{enumerate}
[left=\parindent,itemsep=0pt]
\item
\acli{LC}\emph{:}
	there exists some $\gbound>0$ such that
\begin{align}
\label{eq:LC}
\tag{LC}
\abs{\obj(\pointalt) - \obj(\point)}
	\leq \gbound \norm{\pointalt - \point}
	&\quad
	\text{for all $\point,\pointalt \in \points$}.
\intertext{\item
\acli{LS}\emph{:}
	there exists some $\smooth>0$ such that}
\label{eq:LS}
\tag{LS}
\obj(\pointalt)
	\leq \obj(\point)
	+ \braket{\dvec}{\pointalt - \point}
	+ \frac{\smooth}{2} \norm{\pointalt - \point}^{2}
	&\quad
	\text{for all $\point,\pointalt \in \points$ and all $\dvec\in\subd\obj(\point)$}.
\end{align}
\end{enumerate}

\begin{remark*}
For posterity, we note that the smoothness requirement \eqref{eq:LS} \emph{does not} imply that $\subd\obj(\point)$ is a singleton.
The reason for this more general definition is that we want to concurrently treat problems with smooth and non-smooth objectives, and also feasible domains that are contained in lower-dimensional subspaces of $\vecspace$.%
\footnote{For example, the function $\obj\from\R^{2} \to \R$ with $\obj(\point_{1},0) = \point_{1}$ and $\obj(\point_{1},\point_{2}) = \infty$ for $\point_{2}\neq0$ is perfectly smooth on its domain ($\point_{2} = 0$);
however, $\subd\obj(\point_{1},0) = \setdef{(1,\dvec_{2})}{\dvec_{2}\in\R}$, and this set is never a singleton.}
We also note that we will be mainly interested in cases where the above requirements all \emph{fail} because $\obj$ and/or its derivatives blow up at the boundary of $\points$.
By this token, we will not treat \eqref{eq:LC}/\eqref{eq:LS} as ``blanket assumptions'';
we discuss this in detail in the sequel.
\end{remark*}

\para{The oracle model}

From an algorithmic point of view, we aim to solve \eqref{eq:opt} by using iterative methods that require access to a \acdef{SFO} \citep{Nes04}.
This means that, at each stage of the process, the optimizer can query a black-box mechanism that returns an estimate of the objective's gradient (or subgradient) at the queried point.
Formally, when called at $\point \in \points$, an \ac{SFO} is assumed to return a random (dual) vector $\signal(\point;\sample) \in \dspace$ where $\sample$ belongs to some (complete) probability space $(\samples, \filter,\mathbb{P})$. In practice, the oracle will be called repeatedly at a (possibly) random sequence of points $\state_{\run}\in \points$ generated by the algorithm under study.
Thus, once $\state_{\run}$ has been generated at stage $\run$, the oracle draws an i.i.d.
sample sample $\sample_{\run}\in \samples$ and returns the dual vector:
\begin{equation}
\label{eq:signal}
\signal_{\run}
	\equiv \signal(\state_{\run};\sample_{\run})
	= \nabla \obj(\state_{\run})+\noise_{\run}
\end{equation}
with $\noise_{\run}\in \dspace$ denoting the ``measurement error'' relative to some selection $\nabla\obj(\curr)$ of $\subd\obj(\curr)$.
In terms of measurability, we will write $\filter_{\run}$ for the history (natural filtration) of $\state_{\run}$;
in particular, $\state_{\run}$ is $\filter_{\run}$-adapted, but $\sample_{\run}$, $\signal_{\run}$ and $\noise_{\run}$ are not.
Finally, we will also make the statistical assumption that
\begin{equation}
\label{eq:SFO}
\tag{SFO}
\exof{\noise_{\run} \given \filter_{\run}}
	= 0
	\quad
	\text{and}
	\quad
\dnorm{\noise_{\run}}^{2}
	\leq \sdev^{2}
	\quad
	\text{for all $\run=\running$}
\end{equation}

This assumption is standard in the analysis of parameter-agnostic adaptive methods, \cf \cite{LYC18,KLBC19,WWB19,BL19} and references therein.
For concreteness, we will refer to the case $\sdev = 0$ as deterministic \textendash\ since, in that case, $\noise_{\run} = 0$ for all $\run$.
Otherwise, if $\liminf_{\run} \dnorm{\noise_{\run}} > 0$, the noise will be called \emph{persistent} and the model will be called \emph{stochastic}.

\section{Relative Lipschitz continuity and relative Lipschitz smoothness}
\label{sec:metric}

\subsection{Bregman functions}

We now proceed to describe a flexible template extending the standard Lipschitz continuity and Lipschitz smoothness conditions \textendash\ \eqref{eq:LC} and \eqref{eq:LS} \textendash\ to functions that are possibly singular at the boundary points of $\points$.
The main idea of this extension revolves around the \acdef{NL} framework that was first studied by \citet{BDX11} and then rediscovered independently by \citet{BBT17} and \citet{LFN18}.
The key notion in this setting is that of a suitable ``reference'' \emph{Bregman function} which provides a geometry-adapted measure of divergence on $\points$.
This is defined as follows:

\begin{definition}
\label{def:reg}
A convex \ac{lsc} function $\hreg\from\vecspace\to \R\cup\{\infty\}$ is a
\emph{Bregman function}
on $\points$, if
\begin{enumerate}
[itemsep=0pt,left=\parindent]
\item
$\dom \subd\hreg \subseteq \points \subseteq \dom \hreg$.
\item
The subdifferential of $\hreg$ admits a continuous selection $\nabla \hreg(\point)\in \subd\hreg(\point)$ for all $\point \in \dom \subd\hreg$.
\item
$\hreg$ is strongly convex, \ie there exists some $\hstr>0$ such that
\begin{equation}
\hreg(\pointalt)
	\geq \hreg(\point)
		+ \braket{\nabla \hreg(\point)}{\pointalt - \point}
		+ \frac{\hstr}{2} \norm{\pointalt - \point}^{2}
\end{equation}
for all $\point\in\dom\subd\hreg$, $\pointalt\in\dom\subd\hreg$.
\end{enumerate}
\noindent
The induced \emph{Bregman divergence} of $\hreg$ is then defined for all $\point \in \dom \subd\hreg$, $\pointalt \in \dom \hreg$ as
\begin{equation}
\breg(\pointalt,\point)
	= \hreg(\pointalt)-\hreg(\point)-\braket{\nabla \hreg(\point)}{\pointalt-\point}.
\end{equation}
\end{definition}
\acused{DGF}

\begin{remark*}
The notion of a Bregman function was first introduced by \citet{Bre67}.
Our definition follows \cite{Nes09,NJLS09,JNT11} and leads to the smoothest presentation, but there are variant definitions where $\hreg$ is not necessarily assumed strongly convex, \cf \cite{CL81,CT93,ABM20} and references therein.
\end{remark*}

%

Some standard examples of Bregman functions are as follows:
\begin{itemize}
[topsep=-2pt,left=\parindent]
\item
\textbf{Euclidean regularizer:}
Let $\points$ be a convex subset of $\R^{\vdim}$ endowed with the Euclidean norm $\norm{\cdot}_{2}$.
Then, the \emph{Euclidean regularizer} on $\points$ is defined as $\hreg(\point) = \norm{\point}_{2}^{2}/2$ and the induced Bregman divergence is the standard square distance $\breg(\pointalt,\point) = \norm{\pointalt - \point}_{2}^{2}$ for all $\point,\pointalt\in\points$

\item
\textbf{Entropic regularizer:}
Let $\points = \setdef{\point\in\R_{+}^{\vdim}}{\sum_{i=1}^{\vdim} \point_{i} = 1}$ be the unit simplex of $\R^{\vdim}$ endowed with the $L^{1}$-norm $\norm{\cdot}_{1}$.
Then, the \emph{entropic regularizer} on $\points$ is $\hreg(\point) = \sum_{i} \point_{i} \log\point_{i}$ and the induced divergence is the relative entropy $\breg(\pointalt,\point) = \sum_{i} \pointalt_{i} \log(\pointalt_{\play}/\point_{\play})$ for all $\pointalt\in\points$ $\point\in\relint\points$.

\item
\textbf{Log-barrier:}
Let $\points = \R_{++}^{\vdim}$ denote the (open) positive orthant of $\R^{\vdim}$.
Then, the \emph{log-barrier} on $\points$ is defined as $\hreg(\point) = -\sum_{i=1}^{\vdim} \log\point_{i}$ for all $\point\in\R_{++}^{\vdim}$.
The corresponding divergence is known as the \emph{Itakura-Saito divergence} and is given by $\breg(\point,\pointalt)=\sum_{i=1}^{\vdim} \parens{\point_i/\pointalt_i-\log(\point_{i}/\pointalt_{i}) - 1 }$ \citep{CT93}.
\end{itemize}


\subsection{\Acl{RC}}

With this background in hand, we proceed to discuss how to extend the Lipschitz regularity assumptions of \cref{sec:setup} to account for problems with singular objective functions.
We begin with the notion of \acdef{RC}, as introduced by \citef{Lu19} and extended further in a recent paper by \citef{ZPSH20}:

\begin{definition}
\label{def:RC}
A convex \ac{lsc} function $\obj\from\vecspace \to \R\cup\{\infty\}$
is said to be \emph{relatively continuous} if there exists some $\gbound >0$ such that
\begin{equation}
\tag{RC}
\label{eq:RC}
\obj(\point) - \obj(\pointalt)
	\leq \braket{\nabla \obj(\point)}{\point-\pointalt}
	\leq \gbound\sqrt{2\breg(\pointalt,\point)}
	\quad
	\text{for all $\point\in\dom\hreg, \pointalt\in\dom\subd\hreg$}.
\end{equation}
\end{definition}


In the literature, there have been several extensions of \eqref{eq:LC} to problems with singular objectives.
Below we discuss some of these variants and how they can be integrated in the setting of \cref{def:RC}.

\begin{example}
[$\mathrm{W}{[}\obj,\hreg{]}$ continuity]
\label{ex:Wfh}
This notion intends to single out sufficient conditions for the convergence of ``proximal-like'' methods like \acl{MD}.
Specifically, following \citef{Teb18}, $\obj$ is said to be $\mathrm{W}[\obj,\hreg]$-continuous relative to $\hreg$ on $\points$ (read: ``$\obj$ is weakly $\hreg$-continuous'') if there exists some $\gbound>0$ such that, for all $\run>0$, we have
\begin{equation}
\label{eq:weak}
\tag{W}
\run\braket{\nabla \obj(\point)}{\point-\pointalt}-\breg(\pointalt,\point)
	\leq \frac{\run^{2}}{2}\gbound^{2}
	\quad
	\text{for all $\pointalt \in \dom \hreg$, $\point \in \dom \subd\hreg$}.
\end{equation}
By rearranging the above quadratic polynomial in $t$, we note that its discriminant is $\Delta=\left[ \braket{\nabla \obj(\point)}{\point-\pointalt} \right]^{2} - 2\gbound^{2}\breg(\pointalt,\point)$, so it is immediate to check that \eqref{eq:RC} holds.
\end{example}
\medskip

\begin{example}
[Riemann\textendash Lipschitz continuity]
\label{ex:RLC}
Concurrently to the above, \citef{ABM20} introduced a Riemann-Lipschitz continuity condition extending \eqref{eq:LC} as follows.
Let $\norm{\cdot}_{\point}$ be a family of local norms on $\points$ (possibly induced by an appropriate Riemannian metric), and let $\norm{\dvec}_{\point,\ast}=\max_{\norm{\pointalt}_{\point}\leq 1}\braket{\dvec}{\pointalt}$ denote the corresponding dual norm.
Then, $\obj$ is \emph{Riemann\textendash Lipschitz continuous} relative to $\norm{\cdot}_{\point}$ if there exists some $\gbound >0$ such that:
\begin{equation}
\label{eq:RLC}
\tag{RLC}
\norm{\nabla \obj(\point)}_{\point,\ast}
	\leq \gbound
	\quad
	\text{for all $\point \in \points$}.
\end{equation} 
As we show in the paper's supplement,
\eqref{eq:RLC}$\implies$\eqref{eq:RC} so \eqref{eq:RC} is more general in this regard.
\end{example}

\subsection{\Acl{RS}}

As discussed above, the notion of \acdef{RS} was introduced by \cite{BDX11} and independently rediscovered by \cite{BBT17,LFN18}.
It is defined as follows:

\begin{definition}
\label{def:RS}
A convex \ac{lsc} function $\obj\from\vecspace \to \R\cup\{\infty\}$ is said to be \emph{relatively smooth} if there exists some $\smooth>0$ such that
\begin{equation}
\tag{RS}
\label{eq:RS}
\smooth \hreg - \obj
	\quad
	\text{is convex}.
\end{equation}
\end{definition}

The main motivation behind this elegant definition is the following variational characterizations:

\begin{proposition}
\label{prop:RS-prop}
The following statements are equivalent:
\begin{enumerate}[topsep=-2pt]
\setlength{\itemsep}{0pt}
\setlength{\parsep}{0pt}
\setlength{\parskip}{2pt}
\item
$\obj$ satisfies \eqref{eq:RS}.

\item
$\obj$ satisfies the inequality
\(
\obj(\point)
	\leq \obj(\pointalt)
		+ \braket{\nabla \obj(\pointalt)}{\point-\pointalt}
		+ \smooth \breg(\point,\pointalt),
\)

\item
$\obj$ satisfies the inequality
\(
\braket{\nabla \obj(\point)-\nabla \obj(\pointalt)}{\point-\pointalt}
	\leq \smooth\left[ \breg(\point,\pointalt)+\breg(\pointalt,\point)\right].
\)
\end{enumerate} 
\end{proposition} 

A close variant of \cref{prop:RS-prop} appears in \cite{BDX11,BBT17,LFN18}, so we do not prove it here.
Instead, we discuss below a different extension of \eqref{eq:LS} that turns out to be a special case of \eqref{eq:RS}.
\smallskip


\begin{example}
[Metric smoothness]
\label{ex:MS}
Similar in spirit to \eqref{eq:RLC}, \citef{ABM21} introduced an extension of \eqref{eq:LS} that replaces the global norm $\norm{\cdot}$ with a local norm $\norm{\cdot}_{\point}$, $\point\in\subpoints$.
In particular, given such a norm, we say that $\obj$ is \emph{metrically smooth} (relative to $\norm{\cdot}_{\point}$) if
\begin{equation}
\label{eq:MS}
\tag{MS}
\norm{\nabla \obj(\point)-\nabla \obj(\pointalt)}_{\point,\ast}
	\leq \smooth \norm{\point-\pointalt}_{\pointalt}
	\quad
	\text{for all $\point, \pointalt \in \dom\subd\obj$}.
\end{equation}
An observation that seems to have been overlooked by \cite{ABM21} is that \eqref{eq:MS}$\implies$\eqref{eq:RS}, so \eqref{eq:RS} is more general.
We prove this observation in the appendix.
\end{example}

\subsection{More examples}

Some concrete examples of optimization problems satisfying \eqref{eq:RC}, \eqref{eq:RS} or both (but not their Euclidean counterparts) are Fisher markets \citep{Shm09,BDX11}, Poisson inverse problems \citep{BBT17,ABM20}, support vector machines \citep{Lu19,ZPSH20}, $D$-design \citep{BSTV18,LFN18}, etc.
Because of space constraints, we do not detail these examples here;
however, we provide an in-depth presentation of a Fisher market model in the appendix, along with a series of numerical experiments used to validate the analysis to come.

\section{\Acl{AMD}}
\label{sec:adamir}

We are now in a position to define the proposed \acdef{AMD} method.
In abstract recursive form, \ac{AMD} follows the basic \acl{MD} template
\begin{equation}
\label{eq:MD}
\tag{MD}
\new
	= \prox{\point}{-\step\gvec},
\end{equation}
where 
$\proxmap$ is a generalized Bregman proximal operator induced by $\hreg$ (see below for the detailed definition),
$\gvec$ is a search direction determined by a (sub)gradient of $\obj$ at $\point$,
and
$\step>0$ is a step-size parameter.
We discuss these elements in detail below, starting with the prox-mapping $\proxmap$.

\subsection{The prox-mapping}

Given a Bregman function $\hreg$, its induced \emph{prox-mapping} is defined as
\begin{equation}
\label{eq:prox1}
\prox{\point}{\dvec}
	=\argmin\nolimits_{\pointalt \in \points}\{\braket{\dvec}{\point-\pointalt}
		+ \breg(\pointalt,\point)\}
	\quad
	\text{for all $\point\in\dom\subd\hreg$, $\dvec\in\dspace$},
\end{equation}
where $\breg(\pointalt,\point)$ denotes the Bregman divergence of $\hreg$.
Of course, in order for \eqref{eq:prox1} to be well-defined, the $\argmin$ must be attained in $\dom\subd\hreg$.
Indeed, we have:

\begin{proposition}
\label{prop:well-posed}
The recursion \eqref{eq:MD} satisfies $\new \in \dom\subd\hreg$ for all $\point\in\dom\subd\hreg$ and all $\signal\in\dspace$.
\end{proposition}

To streamline our discussion, we postpone the proof of \cref{prop:well-posed} to \cref{app:Bregman}.
For now, we only note that it implies that the abstract recursion \eqref{eq:MD} is \emph{well-posed}, \ie it can be iterated for all $\run=\running$ to generate a sequence $\state_{\run}\in\points$.

\subsection{The method's step-size}

The next important element of \eqref{eq:MD} is the method's step-size.
In the unconstrained case, a popular adaptive choice is the so-called ``inverse-sum-of-squares'' policy
\begin{equation}
\label{eq:adagrad}
\txs
\step_{\run}
	= 1 \Big/ \sqrt{\sum_{\runalt=\start}^{\run} \dnorm{\nabla\obj(\state_{\runalt})}^{2}},
\end{equation}
where $\state_{\run}$ is the series of iterates produced by the algorithm.
However, in relatively continuous/smooth problems, this definition encounters two crucial issues.
First, because the gradient of $\obj$ is unbounded (even over a bounded domain), the denominator of \eqref{eq:adagrad} may grow at an uncontrollable rate, leading to a step-size policy that vanishes too fast to be of any practical use.
The second is that, if the problem is constrained, the extra terms entering the denominator of $\step_{\run}$ do not vanish as the algorithm approaches a solution, so the \eqref{eq:adagrad} may still be unable to exploit the smoothness of the objective.

We begin by addressing the second issue.
In the Euclidean case, the key observation is that
the difference $\norm{\new - \point}$ must always vanish near a solution (even near the boundary), so we can use it as a proxy for $\nabla \obj(\point)$ in constrained problems.
This idea is formalized by the notion of the \emph{gradient mapping} \cite{Nes04} that can be defined here as
\begin{equation}
\label{eq:grad-map}
\res
	= \norm{\new - \point} \big/ \step.
\end{equation}

On the other hand, in a Bregman setting, the prox-mapping tends to deflate gradient steps, so the norm difference between two successive iterates $\new$ and $\point$ of \eqref{eq:MD} could be very small relative to the oracle signal that was used to generate the update.
As a result, the Euclidean residual \eqref{eq:grad-map} could lead to a disproportionately large step-size that would be harmful for convergence.
For this reason, we consider a gradient mapping that takes into account the Bregman geometry of the method and we set
\begin{equation}
\label{eq:res-Breg}
\res
	= \sqrt{\breg(\point,\new) + \breg(\new,\point)} \big/ \step.
\end{equation}
Obviously, when $\hreg(\point) = (1/2) \twonorm{\point}^{2}$,
we readily recover the definition of the Euclidean gradient mapping \eqref{eq:grad-map}.
In general however, by the strong convexity of $\hreg$, the value of this ``Bregman residual'' exceeds the corresponding Euclidean definition, so the induced step-size exhibits smoother variations that are more adapted to the framework in hand.

\subsection{The \method algorithm}

We are finally in a position to put everything together and define the \acdef{AMD} method.
In this regard, combining the abstract template \eqref{eq:MD} with the Bregman residual and ``inverse-sum-of-squares'' approach discussed above, we will consider the recursive policy
\begin{equation}
\state_{\run+1}
	= \prox{\state_{\run}}{-\step_{\run}\signal_{\run}}
\end{equation}
with $\signal_{\run}$, $\run=\running$, coming from a generic oracle model of the form \eqref{eq:SFO}, and 
with $\step_{\run}$ defined as
\begin{equation}
\label{eq:ada-step}
\tag{\method}
\step_{\run}
	= \frac{1}{\sqrt{\sum_{\runalt=\prestart}^{\run-1}\res^{2}_{\runalt}}}
	\qquad
\text{with}\;\; \res_{\runalt}^{2}=\frac{\breg(\state_{\runalt},\state_{\runalt+1}) + \breg(\state_{\runalt+1},\state_{\runalt})}{\step_{\runalt}^{2}}.
\end{equation}

In the sequel, we will use the term ``\method'' to refer interchangeably to the update $\state_{\run} \gets \state_{\run+1}$ and the specific step-size policy used within.
The convergence properties of \method are discussed in detail in the next two sections (in both deterministic and stochastic problems);
in the supplement, we also perform a numerical validation of the method in the context of a Fisher market model.


%
%

\section{Deterministic analysis and results}
\label{sec:deterministic}

We are now in a position to state our main convergence results for \method.
We begin with the deterministic analysis ($\sdev=0$), treating both the method's ``time-average'' as well as the induced trajectory of query points;
the analysis for the stochastic case ($\sdev > 0$) is presented in the next section.

%

\subsection{Ergodic convergence and rate interpolation}

We begin with the convergence of the method's ``time-averaged'' state, \ie $\avg = (1/\nRuns) \sum_{\run=\start}^{\nRuns} \state_{\run}$.

\begin{theorem}
\label{thm:rates-det}
Let $\state_{\run}$, $\run=\running$, denote the sequence of iterates generated by \method, and let $\iBreg = \breg(\sol,\state_{\start})$.
Then, \method simultaneously enjoys the following guarantees:
\begin{enumerate}
[left=\parindent]
\item
If $\obj$ satisfies \eqref{eq:RC}, we have:
\begin{equation}
\obj(\avg) - \min\obj
	\leq
		\frac
		{\sqrt{2} \gbound \bracks*{\iBreg + 8\gbound^{2}/\res_{0}^{2} + 2\log(1+2\gbound^{2}\nRuns/\res_{0}^{2})} }
		{\sqrt{\nRuns}}
		+ \frac{3\sqrt{2}\gbound + 4\gbound^{2}/\res_{0}^{2}}{\nRuns}.
\end{equation}

\item
If  $\obj$ satisfies \eqref{eq:RS}, we have
\(
\obj(\avg) - \min\obj
	= \bigoh\parens*{\iBreg / \nRuns}.
\)

\item
If  $\obj$ satisfies \eqref{eq:RS} and \eqref{eq:RC}, we have:
\begin{equation}
\obj(\avg) - \min\obj
	\leq
		\bracks*{
			\obj(\state_1)-\min \obj
			+ \parens*{
					2
					+ \frac{8\gbound^{2}}{\res_{\prestart}^{2}}
					+ 2\log\frac{4\smooth^{2}}{\res_{\prestart}^{2}}
				} \smooth 
			 }^{2} 
 		\frac{\iBreg}{\nRuns}.
	\end{equation}
\end{enumerate}
\end{theorem}

\Cref{thm:rates-det} shows that, up to logarithmic factors, \method achieves the min-max optimal bounds for functions in the $\ref{eq:RC} \cup \ref{eq:RS}$ oracle complexity class.%
\footnote{We recall here that, in contrast to \eqref{eq:LS}, the $\bigoh(1/\nRuns)$ rate is optimal in \eqref{eq:RS}, \cf \citef{DTAB19}.}
The key element of the proof (which we detail in \cref{app:adamir}), is the following regret bound:

\begin{proposition}
\label{prop:reg-adamir}
With notation as in \cref{thm:rates-det}, \method enjoys the regret bound
\begin{equation}
\label{eq:reg-adamir}
\sum_{\run=\start}^{\nRuns} \bracks{\obj(\curr) - \obj(\sol)}
	\leq \frac{\iBreg}{\last[\step]}
		+ \frac{\sum_{\run=\start}^{\nRuns} \curr[\step]^{2} \curr[\res]^{2}}{\last[\step]}
		+ \sum_{\run=\start}^{\nRuns} \curr[\step]\curr[\res]^{2}.
\end{equation}
\end{proposition}

The proof of \cref{prop:reg-adamir} hinges on the specific definition of \method's step-size, and the exact functional form of the regret bound \eqref{eq:reg-adamir} plays a crucial role in the sequel.
Specifically, under the regularity conditions \eqref{eq:RC} and \eqref{eq:RS}, we respectively obtain the following key lemmas:

\begin{lemma}
\label{lem:bounded-Breg}
Under \eqref{eq:RC}, the sequence of the Bregman residuals $\curr[\res]$ of \method is bounded as
\(
\curr[\res]^{2}
	\leq 2\gbound^{2}
\)
for all $\run\geq\start$.
\end{lemma}

\begin{lemma}
\label{lem:smooth-Breg}
Under \eqref{eq:RS}, the sequence of the Bregman residuals $\curr[\res]$ of \method is square-summable,
\ie
\(
\sum_{\run} \res_{\run}^{2}
	< \infty.
\)
Consequently, the method's step-size converges to a strictly positive limit $\step_{\infty} > 0$.
\end{lemma}

As we explain below, the boundedness estimate of \cref{lem:bounded-Breg} is necessary to show that the iterates of the method do not explode;
however, without further assumptions, it is not possible to sharpen this bound.
The principal technical difficulty \textendash\ and an important novelty of our analysis \textendash\ is the stabilization of the step-size to a strictly positive limit in \cref{lem:smooth-Breg}.
This property of \method plays a crucial role because the method is not slowed down near a solution.
To the best of our knowledge, there is no comparable result for the step-size of parameter-agnostic methods in the literature.%
\footnote{In more detail, \citef{LYC18}, \citef{LO19} and \citef{KLBC19} establish the summability of a suitable residual sequence to sharpen the $\bigoh(1/\sqrt{\nRuns})$ rate in their respective contexts, but this does not translate to a step-size stabilization result.
Under \eqref{eq:RC}/\eqref{eq:RS}, controlling the method's step-size is of vital importance because the gradients that enter the algorithm may be unbounded even over a bounded domain;
this crucial difficulty does not arise in any of the previous works on adaptive methods for ordinary Lipschitz problems.}

Armed with these two lemmas, we obtain the following series of estimates:
\begin{enumerate}
\item
Under \eqref{eq:RC}, the terms in the \acs{RHS} of \eqref{eq:reg-adamir} can be bounded respectively as
$\bigoh(\gbound \sqrt{\nRuns})$,
$\bigoh(\log(\gbound^{2}\nRuns)\sqrt{\nRuns})$,
and
$\bigoh(\gbound \sqrt{\nRuns})$.
As a result, we obtain an $\tilde\bigoh(1/\sqrt{\nRuns})$ rate of convergence.
\item
Under \eqref{eq:RS}, all terms in the \acs{RHS} of \eqref{eq:reg-adamir} can be bounded as $\bigoh(1)$,
so we obtain an $\bigoh(1/\nRuns)$ convergence rate for $\avg$.
\end{enumerate}
For the details of these calculations (including the explicit constants and logarithmic terms that appear in the statement of \cref{thm:rates-det}), we refer the reader to \cref{app:adamir}.

\subsection{Other modes of convergence}

In complement to the analysis above, we provide below a spinoff result for the method's ``last iterate'', \ie the actual trajectory of queried points.
The formal statement is as follows.

\begin{theorem}
\label{thm:last-iterate}
Suppose that $\obj$ satisfies \eqref{eq:RC} or \eqref{eq:RS}.
Then $\state_{\run}$ converges to $\argmin\obj$.
\end{theorem}

The main idea of the proof (which we detail in the appendix) consists of two steps.
The first key step is to show that, under $\eqref{eq:RC}\cup\eqref{eq:RS}$, the iterates of \method have $\liminf\obj(\curr) = \min\obj$;
we show this in \cref{prop:conv-subsequence1}.
Now, given the existence of a convergent subsequence, the rest of our proof strategy branches out depending on whether $\obj$ satisfies \eqref{eq:RC} or \eqref{eq:RS}.
Under \eqref{eq:RS}, the analysis relies on arguments that involve a quasi-Fejér argument as in \cite{Com01,Bot98}.
However, under \eqref{eq:RC}, the quasi-Fejér property fails, so we prove the convergence of $\curr$ via a novel induction argument that shows that the method's iterates remain trapped within a Bregman neighborhood of $\sol$ if they enter it with a sufficiently small step-size;
we provide the relevant details in \cref{app:Last}.

\para{Non-convex objectives}
We close this section with two remarks on non-convex objectives.
First, \cref{thm:last-iterate} applies verbatim to non-convex objectives $\obj$ satisfying the ``secant condition'' \cite{Bot98,ZMBB+20}
\begin{equation}
\label{eq:secant}
\inf\setdef{\braket{\nabla\obj(\point)}{\point - \sol}}{\sol\in\argmin\obj,\point\in\cpt}
	> 0
\end{equation}
for every closed subset $\cpt$ of $\points$ that is separated by neighborhoods from $\argmin\obj$.
In \cref{app:Last}, our results have all been derived based on this more general condition (it is straightforward to verify that \eqref{eq:secant} always holds for convex functions).

Even more generally, \cref{lem:smooth-Breg} also allows us to derive results for general non-convex problems.
Indeed, the proof of \cref{prop:RS-prop} shows that 
\(
\min_{1\leq\run\leq \nRuns} \res_{\run}^{2}
	= \bigoh(1/\nRuns)
\)
\emph{without} requiring any properties on $\obj$ other than \acl{RS}.
As a result, we conclude that the ``best iterate'' of the method \textendash\ \ie the iterate with the least residual \textendash\ decays as $\bigoh(1/\nRuns)$.
This fact partially generalizes a similar result obtained in \cite{LO19,WWB19} for \adagrad applied to non-convex problems;
however, an in-depth discussion of this property would take us too far afield, so we do not attempt it.

\section{Stochastic analysis}
\label{sec:stochastic}

In this last section, we focus on the stochastic case ($\sdev>0$).
Our main results here are as follows.

\begin{theorem}
\label{thm:stoch}
Let $\state_{\run}$, $\run=\running$, denote the sequence of iterates generated by \method, and let
$\iBreg = \breg(\sol,\state_{\start})$
and
$\gbound_{\sdev} = \gbound + \sdev/\sqrt{\hstr}$.
Then, under \eqref{eq:RC}, we have
\begin{equation}
\label{eq:bound-RC-stoch}
\ex\left[ \obj(\avg)-\obj(\sol)\right]
	\leq (\iBreg + \commbound) \sqrt{\frac{\res_{\prestart}^{2} + 2\gbound_{\sdev}^{2}}{\nRuns}}
\end{equation}
where $\commbound = 8 \gbound_{\sdev}^{2} / \res_{\prestart}^{2} + 2\log(1 + 2\gbound_{\sdev}^{2} \nRuns / \res_{\prestart}^{2})$.
\end{theorem}

Finally, if \eqref{eq:RS} kicks in, we have the sharper guarantee:

\begin{theorem} 
\label{thm:stoch-smooth}
With notation as above, if $\obj$ satisfies \eqref{eq:RS}, \method enjoys the bound
\begin{equation}
\label{eq:bound-RS-smooth}
\exof{\obj(\avg)-\obj(\sol)}
	\leq (2 + \iBreg + \commbound)
		\bracks*{ \frac{\smoothbound}{\nRuns}+\frac{\stochbound\sdev}{\sqrt{\nRuns}} }
\end{equation}
where:
\begin{subequations}
\begin{alignat}{3}
&\text{a\upshape)}
	&\quad
\smoothbound
	&= \res_{\prestart}
		+ 2\bracks{\obj(\state_{\start}) - \min\obj}
		+ \smooth \parens*{2 + 8\gbound_{\sdev}^{2} \big/ \res_{\prestart}^{2} + 2\log(4\smooth^{2}/\res_{\prestart}^{2})}.
	\hspace{1em}
	\\
&\text{b\upshape)}
	&\quad
\stochbound
	&= \sqrt{(4 + 2\commbound)/\hstr}.
\end{alignat}
\end{subequations}
\end{theorem}

The full proof of \cref{thm:stoch,thm:stoch-smooth} is relegated to the supplement, but the key steps are as follows:
\begin{enumerate}
[left=0pt,label={\bfseries Step \arabic*:}]

\item
We first show that, under \eqref{eq:RC}, the method's residuals are bounded as $\res_{\run}^{2} \leq 2\gbound_{\sdev}^{2}$ \as.

\item
With this at hand, the workhorse for our analysis is the following boxing bound for the mean ``weighted'' regret $\sum_{\run=1}^{\nRuns} \exof{\step_{\run}\braket{\nabla \obj(\state_{\run})}{\curr - \sol}}$:
\begin{equation*}
\exof*{\step_{\nRuns} \sum_{\run=\start}^{\nRuns} \bracks{\obj(\state_{\run}) - \obj(\sol)}}
	\leq \exof*{\sum_{\run=1}^{\nRuns} \step_{\run}\braket{\nabla \obj(\state_{\run})}{\curr - \sol}}
	\leq \iBreg
		+\exof*{\sum_{\run=\start}^{\nRuns} \step_{\run}^{2} \res_{\run}^{2}}
\end{equation*}
We prove this bound in the supplement, where we also show that $\exof{\sum_{\run=\start}^{\nRuns} \step_{\run}^{2} \res_{\run}^{2}} = \bigoh(\log\nRuns)$.
\end{enumerate}

At this point the analysis between \cref{thm:stoch,thm:stoch-smooth} branches out.
First, in the case of \cref{thm:stoch}, we show that the method's step-size is bounded from below as $\step_{\run} \geq 1/\sqrt{(\res_{\prestart}^{2} + 2\gbound_{\sdev}^{2}) \run}$;
the guarantee \eqref{eq:bound-RC-stoch} then follows by the boxing bound.
Instead, in the case of \cref{thm:stoch-smooth}, the analysis is more involved and relies crucially on the lower bound $\step_{\run} \geq 1/(\smoothbound + \stochbound\sdev\sqrt{\run})$.
The bound \eqref{eq:bound-RS-smooth} then follows by combining this lower bound for $\step_{\run}$ with the regret boxing bound above.

In the supplement, we also conclude a series of numerical experiments in random Fisher markets that illustrate the method's adaptation properties in an archetypal non-Lipschitz problem.

\section{Concluding remarks}
\label{sec:conclusions}
%
%
Our theoretical analysis confirms that \method concurrently achieves optimal rates of convergence in relatively continuous and relatively smooth problems, both stochastic or deterministic, constrained or unconstrained, and without requiring any prior knowledge of the problem's smoothness/continuity parameters.
These appealing properties open the door to several future research directions, especially regarding the method's convergence properties in non-convex problems.
The ``best-iterate'' discussion of \cref{sec:deterministic} is a first step along the way, but many questions and problems remain open in this direction, especially regarding the convergence of the method's ``last iterate'' in stochastic, non-convex settings.
We defer these questions to future work.

\appendix
\numberwithin{equation}{section}		
\numberwithin{lemma}{section}		
\numberwithin{proposition}{section}		
\numberwithin{theorem}{section}		

\section{Bregman regularizers and mirror maps}
\label{app:Bregman}

Our goal in this appendix is to derive some basic properties for the class of Bregman proximal maps and \acl{MD} methods considered in the main body of our paper.
Versions of the properties that we derive are known in the literature \citep[see \eg][and references therein]{CT93,BecTeb03,Nes09,SS11}.

To begin, we introduce two notions that will be particularly useful in the sequel.
The first is the convex conjugate of a Bregman function $\hreg$, \ie
\begin{equation}
\label{eq:conj}
\hreg^{\ast}(\dpoint)
	= \max_{\point\in\dom\hreg} \{\braket{\dpoint}{\point} - \hreg(\point)\}
\end{equation}
and the associated primal-dual \emph{mirror map} $\mirror\from\dspace\to\dom\subd\hreg$:
\begin{equation}
\label{eq:mirror}
\mirror(\dpoint)
	= \argmax_{\point\in\dom\hreg} \{\braket{\dpoint}{\point} - \hreg(\point)\}
\end{equation}
That the above is well-defined is a consequence of the fact that $\hreg$ is proper, \ac{lsc}, convex and coercive;%
\footnote{The latter holds because $\hreg$ is strongly convex relative to $\norm{\cdot}_{\point}$, and $\norm{\cdot}_{\point}$ has been tacitly assumed bounded from below by a multiple $\mu\norm{\cdot}$ of $\norm{\cdot}$.}
in addition, the fact that $\mirror$ takes values in $\dom\subd\hreg$ follows from the fact that any solution of \eqref{eq:mirror} must necessarily have nonempty subdifferential (see below).
For completeness, we also recall here the definition of the Bregman proximal mapping
\begin{equation}
\label{eq:prox}
\tag{prox}
\prox{\point}{\dvec}
	= \argmin_{\pointalt\in\dom\hreg} \{ \braket{\dvec}{\point - \pointalt} + \breg(\pointalt,\point) \}
\end{equation}
valid for all $\point \in \dom\subd\hreg$ and all $\dvec \in \dspace$.

We then have the following basic lemma connecting the above notions:

\begin{lemma}
\label{lem:Bregman}
Let $\hreg$ be a regularizer in the sense of \cref{def:reg} with $\hstr$-strong convexity modulus. 
Then, for all $\point \in \dom\subd\hreg$ and all $\dvec,\dpoint \in \dspace$ we have:
\begin{enumerate}
\item
$\point=\mirror(\dpoint)\iff \dpoint \in \subd\hreg(\point)$.
\item
\label{eq:domain}
$\new
	= \prox{\point}{\dvec}
	\iff
\nabla\hreg(\point)
	+ \dvec \in \subd\hreg(\point)
	\iff
\new
	= \mirror(\nabla\hreg(\point) + \dvec)$.
\item
Finally, if $\point=\mirror(\dpoint)$ and $\base \in \points$, we get:
\begin{equation}
\label{eq:subeq}
\braket{\nabla\hreg(\point)}{\point-\base}
	\leq \braket{\dpoint}{\point-\base}.
\end{equation}
\end{enumerate}
\end{lemma}

\begin{proof}
For the first equivalence, note that $\point$ solves \eqref{eq:conj} if and only if $0 \in \dpoint - \subd\hreg(\point)$ and hence if and only if $\dpoint \in \subd\hreg(\point)$.
Working in the same spirit for the second equivalence, we get that $\new$ solves \eqref{eq:prox}
if and only if $\nabla\hreg(\point)+\dvec \in \subd\hreg(\new)$ and therefore if and only if 
$\new=\mirror(\nabla\hreg(\point)+\dvec)$.

For our last claim, by a simple continuity argument, it is sufficient to show that the inequality holds for the 
relative interior $\relint\points$ of $\points$ (which, in particular, is contained in $\dom\subd\hreg$).
In order to show this, pick a base point $\base\in\relint\points$, and let
\begin{equation}
\phi(t)
	= \hreg(\point+t(\base-\point))
	- [\hreg(\point)+\braket{\dpoint}{t(\base-\point)}]
	\quad
	\text{for all $t\in[0,1]$}.
\end{equation}
Since, $\hreg$ is strongly convex and $\dpoint\in \subd\hreg(\point)$ due to the first equivalence, it follows
that $\phi(t)\geq 0$ with equality if and only if $t=0$. Since, $\psi(t)=\braket{\nabla\hreg(\point+
t(\base-x))-\dpoint}{\base-\point}$ is a continuous selection of subgradients of $\phi$ and both $\phi$ and
$\psi$ are continuous over $[0,1]$, it follows that $\phi$ is continuously differentiable with $\phi'=\psi$
on $[0,1]$. Hence, with $\phi$ convex and $\phi(t)\geq 0=\phi(0)$ for all $t \in [0,1]$, we conclude that
$\phi'(0)=\braket{\nabla\hreg(x)-\dpoint}{\base-\point}\geq 0$ and thus we obtain the result.
\end{proof}

As a corollary, we have:

\begin{proof}[Proof of \cref{prop:well-posed}]
Our claim follows directly from a tandem application of items (1) and (2) in \cref{lem:Bregman}.
\end{proof}

To proceed, the basic ingredient for establishing connections between Bregman proximal steps is a generalization of the rule of cosines which is known in the literature as the ``three-point identity'' \citep{CT93}.
This will be our main tool for deriving the main estimates for our analysis.
Being more precise, we have the following lemma:

\begin{lemma}
\label{lem:3point}
Let $\hreg$ be a regularizer in the sense of \cref{def:reg}.
Then, for all $\base \in \dom\hreg$ and all $\point,\pointalt \in \dom\subd\hreg$, we have:
\begin{equation}
\breg(\base,\pointalt )= \breg(\base,\point)+\breg(\point,\pointalt)+\braket{\nabla\hreg(\pointalt)-\nabla\hreg(\point)}{\point-\base}.
\end{equation}
\end{lemma}

\begin{proof}
By definition:
\begin{equation}
\begin{aligned}
\breg(\base,\pointalt)
	&= \hreg(\base) - \hreg(\pointalt) - \braket{\nabla\hreg(\pointalt)}{\base - \pointalt}
	\\
\breg(\base,\point)\hphantom{'}
	&= \hreg(\base) - \hreg(\point) - \braket{\nabla\hreg(\point)}{\base - \point}
	\\
\breg(\point,\pointalt)
	&= \hreg(\point) - \hreg(\pointalt) - \braket{\nabla\hreg(\pointalt)}{\point - \pointalt}.
\end{aligned}
\end{equation}
The lemma then follows by adding the two last lines and subtracting the first.
\end{proof}

Thanks to the three-point identity, we obtain the following estimate for the Bregman divergence before and after a \acl{MD} step:

\begin{proposition}
\label{prop:increment}
Let $\hreg$ be a regularizer in the sense of \cref{def:reg}  with strong convexity modulus $\hstr >0$.
Fix some $\base \in \dom\hreg$ and let $\new=\prox{\point}{\dvec}$ for some $\point \in \dom\subd\hreg$ and $\dvec\in \dspace$.
We then have:
\begin{align}
\label{eq:increment}
\breg(\base,\new)
	&\leq \breg(\base,\point)
		- \breg(\new,\point)
		+ \braket{\dvec}{\new-\base}
\shortintertext{and}
\breg(\base,\new)
	&\leq \breg(\base,\point)
		+ \breg(\point,\new)
		- \braket{\dvec}{\point-\base}.
\end{align}
\end{proposition}

\begin{proof}
By the three-point identity established in \cref{lem:3point}, we have:
\begin{equation}
\breg(\base,\point)
	= \breg(\base,\new)
		+ \breg(\new,\point)
		+\braket{\nabla\hreg(\point) - \nabla\hreg(\new)}{\new - \base}
\end{equation}
Rearranging terms then yields:
\begin{equation}
\breg(\base,\new)
	= \breg(\base,\point)
		-\breg(\new,\point)
		+\braket{\nabla\hreg(\new) - \nabla\hreg(\point)}{\new - \base}
\end{equation}
By \eqref{eq:subeq} and the fact that $\new=\prox{\point}{\dvec}$ so $\nabla\hreg(\point)+\dvec \in \subd\hreg(\new)$, the first inequality follows;
the second one is obtained similarly.
\end{proof}

\section{Convergence analysis of \method}
\label{app:adamir}

In this appendix, we will illustrate in detail the convergence analysis of \method, which we present in pseudocode form as \cref{alg:adamir} below.
For ease of presentation we shall divide our analysis, as in the main body of our paper, into two sections: the deterministic and the stochastic one.


\begin{algorithm}[tbp]
\small
\ttfamily
\caption{Adaptive \acl{MD} (\method)}

\begin{algorithmic}[1]
\State
	\textbf{Initialize}
		$\state_{\prestart} \neq \state_{\start} \in \dom\subd\hreg$;
		set $\res_{\prestart} = \bracks{\breg(\state_{\prestart},\state_{\start}) + \breg(\state_{\start},\state_{\prestart})}^{1/2}$
\For{$\run = \running, \nRuns-1$}
	\State
		set $\curr[\step] = \parens[\big]{\sum_{\runalt=\prestart}^{\run-1} \iter[\res]^{2}}^{-1/2}$
		\Comment{step}%
	\State
		get $\curr[\signal] \gets \signal(\state_{\run};\sample_{\run})$
		\Comment{feedback}%
	\State
		set $\next = \prox{\curr}{-\curr[\step]\curr[\signal]}$
		\Comment{Bregman step}%
	\State
		set $\curr[\res] = \bracks{\breg(\curr,\next) + \breg(\next,\curr)}^{1/2} / \curr[\step]$
		\Comment{Bregman residual}%
\EndFor
	\State
		\Return $\avg \subs (1/\nRuns) \sum_{\run=\start}^{\nRuns} \curr$
		\Comment{candidate solution}%
\end{algorithmic}
\label{alg:adamir}
\end{algorithm}


\subsection{The deterministic case}
We begin with the proof of \cref{lem:bounded-Breg} which provides an upper bound to the Bregman residuals generated by \method:


\begin{proof}[Proof of \cref{lem:bounded-Breg}]
By the definition of the Bregman proximal step in \eqref{eq:MD} and \cref{prop:increment}, we have:
\begin{align}
\breg(\curr,\next) + \breg(\next,\curr)
	&= \braket{\nabla\hreg(\curr)-\nabla\hreg(\next)}{\curr-\next}
	\notag\\
	&\leq \curr[\step]\braket{\curr[\signal]}{\curr-\next}.
\end{align}
Hence, by applying the \eqref{eq:RC} condition of the objective we get:
\begin{align}
\breg(\curr,\next)
	+ \breg(\next,\curr)
	&\leq \curr[\step] \gbound \sqrt{2\breg(\next,\curr)}
	\notag\\
	&\leq \curr[\step] \gbound  \sqrt{2\left[\breg(\next,\curr)+\breg(\curr,\next)\right]}
\end{align}

We thus get:
\begin{equation}
\breg(\curr,\next) + \breg(\next,\curr)
	\leq {2\curr[\step]^{2}\gbound^{2}}.
\end{equation}
Hence, by the definition \eqref{eq:res-Breg} of $\curr[\res]^2$, we conclude that
\begin{equation}
\curr[\res]^{2}
	= \frac{\breg(\curr,\next) + \breg(\next,\curr)}{\curr[\step]^{2}}
	\leq {2\gbound^{2}}.
\qedhere
\end{equation}
\end{proof}

\begin{proof}[Proof of \cref{lem:smooth-Breg}]
Since the adaptive step-size policy $\step_{\run}$ is decreasing and bounded from below $(\step_{\run})\geq 0$ we get that its limit exist,\ie
\begin{equation}
\lim_{\run \to +\infty}\step_{\run}=\gamma_{\infty}\;\;\text{for some}\;\;\gamma_{\infty}\geq 0
\end{equation} 
Assume that $\gamma_{\infty}=0$. By \cref{prop:RS-prop}, we obtain:
\begin{align}
\obj(\state_{\run+1})
	\leq \obj(\state_{\run})
	&+ \braket{\nabla \obj(\state_{\run})}{\state_{\run+1}-\state_{\run}}+{\smooth}\breg(\state_{\run+1},\state_{\run})
	\notag\\
	\leq \obj(\state_{\run})
	&-\frac{1}{\step_{\run}}\breg(\state_{\run},\state_{\run+1})
	\notag\\
	&- \frac{1}{\step_{\run}}\breg(\state_{\run+1},\state_{\run})+\smooth\left[\breg(\state_{\run},\state_{\run+1})+\breg(\state_{\run+1},\state_{\run})\right]
\end{align}
whereas by recalling the definition of the residuals \eqref{eq:ada-step} the above can be rewritten as follows:
\begin{align}
\obj(\state_{\run+1})
	\leq \obj(\state_{\run})-\step_{\run}\res_{\run}^{2}+\smooth\step_{\run}^{2}\res_{\run}^{2}
	=\obj(\state_{\run})-\frac{1}{2}\step_{\run}\res_{\run}^{2}-\frac{1}{2}\step_{\run}\res_{\run}^{2}+\smooth\step_{\run}^{2}\res_{\run}^{2}
\end{align}
Moreover, by rearranging and factorizing the common term $\step_{\run}\res_{\run}^{2}$ we get:
\begin{equation}
\frac{1}{2}\step_{\run}\res_{\run}^{2}\leq \obj(\state_{\run})-\obj(\state_{\run+1})+\step_{\run}\res_{\run}^{2}\left[{\smooth}\step_{\run}-\frac{1}{2}\right]
\end{equation}
Now,  by combing that $\left[{\smooth}\step_{\run}-\frac{1}{2}\right]\leq 0$ for $\step_{\run}\leq 1/2\smooth$ and the fact that $\step_{\run}$ converges to $0$ by assumption, we get that there exists some $\run_{0}\in \N$ such that:
\begin{equation}
\left[{\smooth}\step_{\run}-\frac{1}{2}\right]\leq 0\;\;\text{for all}\;\run > \run_0
\end{equation}
Hence, by telescoping for $\run=1,2,\dotsc, \nRuns$ for sufficiently large $\nRuns$, we have
\begin{align}
\frac{1}{2}\sum_{\run=1}^{\nRuns}\step_{\run}\res_{\run}^{2}
	&\leq \obj(\state_{1})-\obj(\state_{\nRuns+1})+\sum_{\run=1}^{\run_0}\left[{\smooth}\step_{\run}-\frac{1}{2}\right]\step_{\run}\res_{\run}^{2}
	\notag\\
	&\leq \obj(\state_{1})-\min_{\point \in \points}\obj(\point)+\sum_{\run=1}^{\run_0}\left[{\smooth}\step_{\run}-\frac{1}{2}\right]\step_{\run}\res_{\run}^{2}
\end{align}
Now, by applying the (LHS) of \cref{lem:ineq} we get:
\begin{equation}
\frac{1}{2}\left[\frac{1}{\step_{\nRuns}}-\res_{0}\right]\leq\frac{1}{2}\sqrt{\res_{0}^{2}+\sum_{\run=1}^{\nRuns-1}\step_{\run}\res_{\run}^{2}}\leq \sum_{\run=1}^{\nRuns}\step_{\run}\res_{\run}^{2}\leq \obj(\state_1)-\min_{\point \in \points}\obj(\point)+\sum_{\run=1}^{\run_0}\left[{\smooth}\step_{\run}-\frac{1}{2}\right]\step_{\run}\res_{\run}^{2}
\end{equation}
Now, since $\step_{\run}\to 0$ we get that $1/\step_{\run}\to +\infty$ and hence the above  yields that $+\infty\leq \obj(\state_1)-\min_{\point \in \points}\obj(\point)+\sum_{\run=1}^{\run_0}\left[\frac{\smooth}{\hstr}\step_{\run}-\frac{1}{2}\right]\step_{\run}\res_{\run}^{2}$; a contradiction. Therefore we get that:
\begin{equation}
\lim_{\run \to +\infty}\step_{\run}=\gamma_{\infty}>0
\end{equation}
Moreover, by recalling the definition of the adaptive step-size policy $\step_{\run}$:
\begin{equation}
\step_{\run}=\frac{1}{\sqrt{\res_{0}^{2}+\sum_{\runalt=1}^{\run-1}\res_{\runalt}^{2}}}
\end{equation}
whereas after rearranging we obtain:
\begin{equation}
\sum_{\runalt=1}^{\run-1}\res_{\runalt}^{2}=\frac{1}{\step_{\run}^{2}}-\res_{0}^{2}
\end{equation}
and therefore by taking limit on both sides we obtain:
\begin{equation}
\sum_{\run=1}^{+\infty}\res_{\run}^{2}=\lim_{\run \to +\infty}\sum_{\runalt=1}^{\run-1}\res_{\runalt}^{2}=\lim_{\run \to +\infty}\frac{1}{\step_{\run}^{2}}-\res_{0}^{2}=\frac{1}{\gamma_{\infty}^{2}}-\res_{0}^{2}<+\infty
\end{equation}
and hence the result follows.
\end{proof}

We proceed by providing an upper bound in terms of the Bregman divergence for the distance of the algorithm's iterates from a solution of \eqref{eq:opt}:

\begin{lemma}
\label{lem:bound-2}
For all $\sol\in\sols$, the iterates of \cref{alg:adamir} satisfy the bound
\begin{equation}
\breg(\sol,\curr)
	\leq \breg(\sol,\init)
		+ \sum_{\runalt=\start}^{\nRuns} \iter[\step]^2 \iter[\res]^{2}.
\end{equation}
\end{lemma}
\begin{proof}
By the second part of \cref{prop:increment}, we have:
\begin{align}
\breg(\sol,\nextiter)
	&\leq \breg(\sol,\iter)
		- \curr[\step]\braket{\curr[\signal]}{\curr - \sol}
		+ \breg(\iter,\nextiter)
	\notag\\
	&\leq \breg(\sol,\iter)
		+ \breg(\iter,\nextiter)
	\notag\\
	&\leq \breg(\sol,\iter)
		+ \breg(\nextiter,\iter)
		+ \breg(\iter,\nextiter)
\end{align}
Thus, by telescoping through $\runalt = \running,\run$, we obtain:
\begin{align}
\breg(\sol,\curr)
	&\leq \breg(\sol,\init)
		+ \sum_{\runalt=\start}^{\run} \bracks*{\breg(\iter,\nextiter) + \breg(\nextiter,\iter)}
	\notag\\
	&\leq \breg(\sol,\init)
		+\sum_{\runalt=\start}^{\nRuns} \bracks*{\breg(\iter,\nextiter) + \breg(\nextiter,\iter)}
	\notag\\
	&= \breg(\sol,\init)
		+ \sum_{\runalt=\start}^{\nRuns} \iter[\step]^{2} \iter[\res]^{2}
\end{align}
where the last equality follows from the definition \eqref{eq:res-Breg} of $\curr[\res]$.
\end{proof}

With these intermediate results at our disposal, we are finally in a position to prove the core estimate \eqref{eq:reg-adamir} for \method:


\begin{proof}[Proof of \cref{prop:reg-adamir}]
By the convexity of $\obj$, the definition of the Bregman proximal step in \cref{alg:adamir} and \cref{prop:increment}, we have:
\begin{align}
\obj(\curr) - \obj(\sol)
	\leq \braket{\curr[\signal]}{\curr-\sol}
	\leq \frac{1}{\curr[\step]} \braket{\nabla\hreg(\curr) - \nabla\hreg(\next)}{\curr - \sol}.
\end{align}
Hence, by applying again the three-point identity (\cref{lem:3point}), we obtain:
\begin{align}
\obj(\curr) - \obj(\sol)
	&\leq \frac{\breg(\sol,\curr) - \breg(\sol,\next)}{\curr[\step]}
		+ \frac{\breg(\curr,\next)}{\curr[\step]}
\notag\\
&\leq \frac{\breg(\sol,\curr)-\breg(\sol,\next)}{\curr[\step]}+\frac{\breg(\curr,\next)+\breg(\next,\curr)}{\curr[\step]}
\notag\\
&= \frac{\breg(\sol,\curr)-\breg(\sol,\next)}{\curr[\step]}+ \curr[\step]\curr[\res]^{2}
\end{align}
where the last equality follows readily from the definition \eqref{eq:res-Breg} of $\curr[\res]$.
Therefore, by summing through $\run = \running, \nRuns$, we obtain:
\begin{equation}
\label{eq:est1}
\sum_{\run=\start}^{\nRuns} \bracks{\obj(\curr) - \obj(\sol)}
	\leq \frac{\breg(\sol,\init)}{\init[\step]}
	+ \sum_{\run=\afterstart}^{\nRuns}
		\bracks*{\frac{1}{\curr[\step]} - \frac{1}{\prev[\step]}} \breg(\sol,\curr)
	+\sum_{\run=\start}^{\nRuns} \curr[\step] \curr[\res]^{2}.
\end{equation}
Now, by \cref{lem:bound-2}, the second term on the \ac{RHS} of \eqref{eq:est1} becomes:
\begin{align}
\sum_{\run=\afterstart}^{\nRuns}
	\bracks*{\frac{1}{\curr[\step]} - \frac{1}{\prev[\step]}} \breg(\sol,\curr)
	&\leq \sum_{\run=\afterstart}^{\nRuns}
		\bracks*{\frac{1}{\curr[\step]} - \frac{1}{\prev[\step]}}
		\parens*{\breg(\sol,\init) + \sum_{\runalt=\start}^{\nRuns}\iter[\step]^{2}\iter[\res]^{2}}
	\notag\\
	&\leq \frac{\breg(\sol,\init)}{\step_{\nRuns}}
		- \frac{\breg(\sol,\init)}{\init[\step]}
		+ \sum_{\runalt=\start}^{\nRuns} \iter[\step]^{2} \iter[\res]^{2}
			\cdot \sum_{\run=\start}^{\nRuns}
				\bracks*{\frac{1}{\curr[\step]} - \frac{1}{\prev[\step]}}
	\notag\\
&\leq \frac{\breg(\sol,\init)}{\last[\step]}-\frac{\breg(\sol,\init)}{\init[\step]}+\frac{\sum_{\run=1}^{\nRuns}\curr[\step]^{2}\curr[\res]^{2}}{\last[\step]}.
\end{align}
Hence, by combining the above with \eqref{eq:est1}, our claim follows.
\end{proof}


With the regret bound \eqref{eq:reg-adamir} at our disposal, we may finally proceed with the proof of our main result concerning the universality of \method, \ie \cref{thm:rates-det} :

\begin{proof}[Proof of \cref{thm:rates-det}]
Repeating the statement of \cref{prop:reg-adamir}, the iterate sequence $\curr$ generated by \method enjoys the bound:
\begin{equation}
\tag{\ref*{eq:reg-adamir}}
\sum_{\run=1}^{\nRuns}\bracks{\obj(\curr)-\obj(\sol)}\leq \frac{\breg(\sol,\init)}{\last[\step]}+\frac{\sum_{\run=1}^{\nRuns}\curr[\step]^2\curr[\res]^{2}}{\last[\step]}+\sum_{\run=1}^{\nRuns}\curr[\step]\curr[\res]^{2}
\end{equation}
We now proceed to bound each term on the \ac{RHS} of \eqref{eq:reg-adamir} from above.
We consider three separate cases, first only under \eqref{eq:RC},then under \eqref{eq:RS} and finally when \eqref{eq:RC} and \eqref{eq:RS} holds.

\para{Case 1}

We begin with problems satisfying \eqref{eq:RC}.
\begin{itemize}
\item
For the first term, \cref{lem:bounded-Breg} gives:
\begin{equation}
\frac{\breg(\sol,\init)}{\last[\step]}
	= \breg(\sol,\init) \sqrt{\sum_{\run=\prestart}^{\nRuns-1} \curr[\res]^2}
	\leq \breg(\sol,\init) \sqrt{{2\gbound^{2}\nRuns}}.
\end{equation}
\item
For the second term, we have:
\begin{equation}
\sum_{\run=\start}^{\nRuns} \curr[\step]^{2} \curr[\res]^{2}
	\leq \sum_{\run=\start}^{\nRuns}
		\frac{\curr[\res]^{2}}{\sum_{\runalt=\prestart}^{\run-1} \iter[\res]^{2}}
	= \sum_{\run=\start}^{\nRuns}
		\frac{\curr[\res]^{2}}{\res_{\prestart}^{2} + \sum_{\runalt=1}^{\run-1} \iter[\res]^{2}}.
\end{equation}
Hence, by \cref{lem:bounded-Breg,lem:logarithmic-3}, we get:
\begin{align}
\sum_{\run=1}^{\nRuns}\curr[\step]^{2}\curr[\res]^{2}
	&\leq 2
		+ \frac{8\gbound^{2}}{\res_{\prestart}^{2}}
		+ 2\log\parens*{1+\sum_{\run=\start}^{\nRuns-1} \frac{\curr[\res]^2}{\res_{\prestart}^{2}}}
	\notag\\
	&= 2
		+ \frac{8\gbound^{2}}{ \res_{\prestart}^{2}}
		+ 2\log\parens*{\sum_{\run=0}^{\nRuns-1} \frac{\curr[\res]^2}{\res_{\prestart}^{2}}}
	\notag\\
	&\leq 2
		+ \frac{8\gbound^{2}}{ \res_{\prestart}^{2}}
		+ 2\log\frac{2\gbound^{2}\nRuns}{\res_{\prestart}^{2}}.
\end{align}
\item
Finally, for the third term, we get:
\begin{equation}
\sum_{\run=\start}^{\nRuns} \curr[\step]\curr[\res]^{2}
	= \sum_{\run=\start}^{\nRuns}
		\frac{\curr[\res]^{2}}{\sqrt{\sum_{\runalt=\prestart}^{\run-1} \curr[\res]^2}}
	= \sum_{\run=\start}^{\nRuns}
		\frac{\curr[\res]^{2}}{\sqrt{\res_{\prestart}^{2}
		+ \sum_{\runalt=\start}^{\run-1} \curr[\res]^2}}.
\end{equation}
Hence, \cref{lem:bounded-Breg,lem:ineq} again yield:
\begin{align}
\sum_{\run=\start}^{\nRuns} \curr[\step] \curr[\res]^{2}
	&\leq \frac{4\gbound^{2}}{ \res_{\prestart}}
		+ {3\sqrt{2}\gbound}
		+ 3\sqrt{\res_{\prestart}^{2}
		+ \sum_{\run=\start}^{\nRuns-1}\curr[\res]^2}
	\notag\\
	&\leq \frac{4\gbound^{2}}{\res_{\prestart}}
		+ {3\sqrt{2}\gbound}
		+ 3\sqrt{\sum_{\run=\prestart}^{\nRuns-1} \curr[\res]^2}
	\notag\\
	&\leq \frac{4\gbound^{2}}{\res_{\prestart}}
		+ {3\sqrt{2}\gbound}
		+ 3\sqrt{2\gbound^{2}\nRuns}.
\end{align}
\end{itemize}
The claim of \cref{thm:rates-det} then follows by combining the above within the regret bound \eqref{eq:reg-adamir}.

\para{Case 2}

We now turn to problems satisfying \eqref{eq:RS}. Recalling \cref{lem:smooth-Breg}, we shall revisit the terms of \eqref{eq:reg-adamir}. In particular, we have:

\begin{itemize}
\item
For the first term,  we have:
\begin{equation}
\frac{\breg(\sol,\init)}{\last[\step]}
	= \breg(\sol,\init) \sqrt{\sum_{\run=\prestart}^{\nRuns-1} \curr[\res]^2}
	\leq \frac{\breg(\sol,\init)}{\step_{\infty}}
\end{equation}

\item
For the second term, we have:
\begin{equation}
\sum_{\run=\start}^{\nRuns} \curr[\step]^{2} \curr[\res]^{2}
	\leq \frac{1}{\res_{\prestart}^{2}}\sum_{\run=\start}^{\nRuns}\res_{\run}^{2}\leq \frac{1}{\res_{\prestart}^{2}\step_{\infty}^{2}}-1
\end{equation}
\item
Finally, for the third term, we get:
\begin{equation}
\sum_{\run=\start}^{\nRuns} \curr[\step]\curr[\res]^{2}
	\leq \frac{1}{\res_{\prestart}}\sum_{\run=\start}^{\nRuns}\res_{\run}^{2}\leq \frac{1}{\res_{\prestart}\step_{\infty}^{2}}-\res_{\prestart}
	\end{equation}
\end{itemize}
Combining all the above, the result follows.

\para{Case 3}
Finally, we consider objectives where \eqref{eq:RC} and \eqref{eq:RS} hold simultaneously. Now, by working in the same spirit as in the proof of \cref{lem:smooth-Breg} we get:
\begin{equation}
\frac{1}{2}\step_{\run}\res_{\run}^{2}\leq \obj(\state_{\run})-\obj(\state_{\run+1})+\step_{\run}\res_{\run}^{2}\left[{\smooth}\step_{\run}-\frac{1}{2}\right]
\end{equation}

which after telescoping $\run=\start, \dotsc ,\nRuns$ it becomes:
\begin{equation}
\frac{1}{2}\sum_{\run=\start}^{\nRuns}\step_{\run}\res_{\run}^{2}\leq \obj(\state_{1})-\min_{\point \in \points}\obj(\point)+\sum_{\run=\start}^{\nRuns}\step_{\run}\res_{\run}^{2}\left[{\smooth}\step_{\run}-\frac{1}{2}\right]
\end{equation}

Now, after denoting:
\begin{equation}
\run_{0}=\max\{\run \in \N: \start \leq \run \leq \nRuns \;\; \text{such that} \;\;\step_{\run}\geq \frac{1}{2\smooth}\}
\end{equation}
and decomposing the sum we get:
\begin{align}
\frac{1}{2}\sum_{\run=\start}^{\nRuns}\step_{\run}\res_{\run}^{2}
	&\leq \obj(\state_{1})-\min_{\point \in \points}\obj(\point)
		+ \sum_{\run=\start}^{\run_{0}}\step_{\run}\res_{\run}^{2}\left[{\smooth}\step_{\run}-\frac{1}{2}\right]
		+\sum_{\run=\run_{0}+1}^{\nRuns}\step_{\run}\res_{\run}^{2}\left[{\smooth}\step_{\run}-\frac{1}{2}\right]
	\notag\\
	&\leq \obj(\state_{1})-\min_{\point \in \points}\obj(\point)
		+\sum_{\run=\start}^{\run_{0}}\step_{\run}\res_{\run}^{2}\left[{\smooth}\step_{\run}-\frac{1}{2}\right]
	\notag\\
	&\leq  \obj(\state_{1})-\min_{\point \in \points}\obj(\point)+\smooth\sum_{\run=\start}^{\run_{0}}\step_{\run}^{2}\res_{\run}^{2}
\end{align} 
On the other  hand, by applying \cref{lem:logarithmic-3}, we  have:

\begin{align}
\sum_{\run=1}^{\run_{0}}\curr[\step]^{2}\curr[\res]^{2}
	&\leq 2
		+ \frac{8\gbound^{2}}{\res_{\prestart}^{2}}
		+ 2\log\parens*{1+\sum_{\run=\start}^{\run_{0}-1} \frac{\curr[\res]^2}{\res_{\prestart}^{2}}}
	\notag\\
	&= 2
		+ \frac{8\gbound^{2}}{ \res_{\prestart}^{2}}
		+ 2\log\parens*{\frac{1}{\res_{\prestart}^{2}}\left[\res_{\prestart}^{2}+\sum_{\run=1}^{\run_{0}-1}\res_{\run}^{2}\right]}
			\notag\\
	&= 2
		+ \frac{8\gbound^{2}}{ \res_{\prestart}^{2}}
		+ 2\log\frac{1}{\res_{\prestart}^{2}\step_{\run_{0}}^{2}}
\end{align}
and by definition of $\run_{0}$ we get:
\begin{equation}
\sum_{\run=1}^{\run_{0}}\curr[\step]^{2}\curr[\res]^{2}\leq 2
		+ \frac{8\gbound^{2}}{ \res_{\prestart}^{2}}+2\log\frac{4\smooth^{2}}{\res_{\prestart}^{2}}.
		\end{equation}
which yields:
\begin{equation}
\sum_{\run=1}^{\nRuns}\step_{\run}\res_{\run}^{2}\leq \obj(\state_1)-\min_{\point \in \points}\obj(\point)+\smooth \left[ 2
		+ \frac{8\gbound^{2}}{ \res_{\prestart}^{2}}+2\log\frac{4\smooth^{2}}{\res_{\prestart}^{2}}\right]
\end{equation}
The result then follows by plugging in the above bounds in \eqref{eq:reg-adamir}.
\end{proof}

\subsection{The stochastic case}

In this appendix, we shall provide the stochastic part of our analysis. We start by providing an intermediate lemma concerning the class of \eqref{eq:RC} objectives.

\begin{lemma}
\label{lem:bound-stoch}
Assume that $\obj$ satisfies \eqref{eq:RC} and $\state_{\run}$ are the \method iterates run with feedback of the form \eqref{eq:SFO}. Then, the sequence of the residuals $\res_{\run}^{2}$ is bounded with probability
$1$. In particular, we have: 
\begin{equation}
\res_{\run}^{2}\leq \tilde{\gbound}^{2}=\left[\sqrt{2}\gbound +\sqrt{\frac{2}{\hstr}}\sdev\right]^{2}
 \;\;\text{for all $\run=\start,2,\dotsc$\; almost surely}
\end{equation}
\end{lemma}

\begin{proof}
By working in the same spirit, we get that:
\begin{equation}
\breg(\state_{\run},\state_{\run+1})+\breg(\state_{\run+1},\state_{\run})\leq \step_{\run} \braket{\signal_{\run}}{\state_{\run}-\state_{\run+1}}
\end{equation}
and by recalling that:
\begin{equation}
\signal_{\run}=\nabla \obj(\state_{\run})+\noise_{\run}
\end{equation}
 we get with probability $1$: 
\begin{align}
\breg(\state_{\run},\state_{\run+1})+\breg(\state_{\run+1},\state_{\run})
	&\leq \step_{\run}\left[\braket{\nabla \obj(\state_{\run})}{\state_{\run}-\state_{\run+1}}+\braket{\noise_{\run}}{\state_{\run}-\state_{\run+1}}\right]
	\notag\\
	&\leq \step_{\run}\left[\gbound\sqrt{2\breg(\state_{\run+1},\state_{\run})} +\dnorm{\noise_{\run}}\norm{\state_{\run}-\state_{\run+1}}\right]
\end{align}
with the second inequality being obtained by \eqref{eq:RC}. Now,  by invoking the strong convexity assumption of $\hstr$, the (LHS) of the above becomes:
\begin{multline}
\step_{\run}\left[\gbound\sqrt{2\breg(\state_{\run+1},\state_{\run})} +\dnorm{\noise_{\run}}\norm{\state_{\run}-\state_{\run+1}}\right]\leq \step_{\run}[\gbound\sqrt{2(\breg(\state_{\run+1},\state_{\run})+\breg(\state_{\run},\state_{\run+1}))} 
\\
+\dnorm{\noise_{\run}}\sqrt{\frac{2}{\hstr}(\breg(\state_{\run+1},\state_{\run})+\breg(\state_{\run},\state_{\run+1})})]
\end{multline}
which in turn yields:
\begin{equation}
\breg(\state_{\run},\state_{\run+1})+\breg(\state_{\run+1},\state_{\run})\leq\step_{\run}\sqrt{\breg(\state_{\run+1},\state_{\run})+\breg(\state_{\run},\state_{\run+1})}\left[\sqrt{2}\gbound +\sqrt{\frac{2}{\hstr}}\dnorm{\noise_{\run}}\right]
\end{equation}
Therefore, we get:
\begin{equation}
\breg(\state_{\run},\state_{\run+1})+\breg(\state_{\run+1},\state_{\run})\leq \step_{\run}^{2}\left[\sqrt{2}\gbound +\sqrt{\frac{2}{\hstr}}\dnorm{\noise_{\run}}\right]^{2}
\end{equation}
and by \acdef{SFO} we get with probability $1$:
\begin{equation}
\breg(\state_{\run},\state_{\run+1})+\breg(\state_{\run+1},\state_{\run})\leq \step_{\run}^{2}\left[\sqrt{2}\gbound +\sqrt{\frac{2}{\hstr}}\sdev\right]^{2}
\end{equation}
or equivalently,
\begin{equation}
\res_{\run}^{2}=\frac{\breg(\state_{\run},\state_{\run+1})+\breg(\state_{\run+1},\state_{\run})}{\step_{\run}^{2}}\leq \left[\sqrt{2}\gbound +\sqrt{\frac{2}{\hstr}}\sdev\right]^{2}
\end{equation}
and the result follows.
\end{proof}

Finally, we provide the proof of the first. theorem for the stochastic setting. 

\begin{proof}[Proof of \cref{thm:stoch}]
By the second part of \cref{prop:increment}, we have:
\begin{align}
\breg(\sol,\state_{\run+1})
	&\leq \breg(\sol,\state_{\run})
		- \curr[\step]\braket{\curr[\signal]}{\curr - \sol}
		+ \breg(\state_{\run},\state_{\run+1})
	\notag\\
	&\leq \breg(\sol,\state_{\run})
	- \curr[\step]\braket{\curr[\signal]}{\curr - \sol}+ \breg(\state_{\run+1},\state_{\run})
		+ \breg(\state_{\run},\state_{\run+1})
					\notag\\
	&\leq \breg(\sol,\state_{\run})
		-\curr[\step]\braket{\curr[\signal]}{\curr - \sol} +\step_{\run}^{2}\res_{\run}^{2}
				\end{align}
which yields after rearranging and summing $\run=\start, \dotsc ,\nRuns$: 
\begin{equation}
\sum_{\run=\start}^{\nRuns}\curr[\step]\braket{\curr[\signal]}{\curr - \sol}\leq \breg(\sol,\state_1)+\sum_{\run=\start}^{\nRuns}\step_{\run}^{2}\res_{\run}^{2}
\end{equation}
and by recalling that $\signal_{\run}=\nabla \obj(\state_{\run})+\noise_{\run}$ and taking expectations on both sides we get:
\begin{equation}
\label{eq:bound1}
\ex\left[\sum_{\run=1}^{\nRuns}\step_{\run}\braket{\nabla \obj(\state_{\run})}{\curr - \sol}\right]\leq \breg(\sol,\state_1)+\ex\left[ \sum_{\run=1}^{\nRuns}\step_{\run}\braket{\noise_{\run}}{\curr - \sol}\right]+\ex\left[ \sum_{\run=\start}^{\nRuns}\step_{\run}^{2}\res_{\run}^{2}\right]
\end{equation}
First, we shall the (LHS) from below. In particular,  we have by convexity:
\begin{equation}
\ex\left[\sum_{\run=1}^{\nRuns}\step_{\run}\braket{\nabla \obj(\state_{\run})}{\curr - \sol}\right]\geq \ex\left[\sum_{\run=1}^{\nRuns}\step_{\run}(\obj(\state_{\run})-\obj(\sol))\right]
\end{equation}
Moreover, by denoting $\tilde{\gbound}^{2}= \left[\sqrt{2}\gbound +\sqrt{\frac{2}{\hstr}}\sdev\right]^{2}$  we have with probability $1$:
\begin{align}
\sum_{\run=1}^{\nRuns}\step_{\run}(\obj(\state_{\run})-\obj(\sol)
	&= \sum_{\run=1}^{\nRuns}\frac{1}{\sqrt{\res_{0}^{2}+\sum_{\runalt=\start}^{\run-1}\res_{\runalt}^{2}}}(\obj(\state_{\run})-\obj(\sol)
	\notag\\
	&\geq \sum_{\run=1}^{\nRuns}\frac{1}{\sqrt{\res_{0}^{2}+\tilde{\gbound}^{2}\run}}(\obj(\state_{\run})-\obj(\sol))
	\notag\\
	&\geq \sum_{\run=1}^{\nRuns}\frac{1}{\sqrt{(\res_{0}^{2}+\tilde{\gbound}^{2})\run}}(\obj(\state_{\run})-\obj(\sol)
	\notag\\
	&\geq \frac{1}{\sqrt{(\res_{0}^{2}+\tilde{\gbound}^{2})\nRuns}} \sum_{\run=1}^{\nRuns}(\obj(\state_{\run})-\obj(\sol)
\end{align}
with the second inequality being obtained by \cref{lem:bound-stoch}. Hence, we get:
\begin{equation}
\ex\left[\sum_{\run=1}^{\nRuns}\step_{\run}\braket{\nabla \obj(\state_{\run})}{\curr - \sol}\right]\geq \frac{1}{\sqrt{(\res_{0}^{2}+\tilde{\gbound}^{2})\nRuns}} \ex\left[\sum_{\run=1}^{\nRuns}(\obj(\state_{\run})-\obj(\sol))\right]
\end{equation}
We now turn our attention towards to the (LHS). In particular, we shall bound each term individually from above.
\begin{itemize}
\item
For the term $\ex\left[ \sum_{\run=1}^{\nRuns}\step_{\run}\braket{\noise_{\run}}{\curr - \sol}\right]$:
\begin{align}
\ex\left[ \sum_{\run=1}^{\nRuns}\step_{\run}\braket{\noise_{\run}}{\curr - \sol}\right]
	&=\sum_{\run=\start}^{\nRuns}\ex\left[\step_{\run}\braket{\noise_{\run}}{\curr - \sol} \right]
	\notag\\
	&=\sum_{\run=\start}^{\nRuns}\ex\left[\ex\left[\step_{\run}\braket{\noise_{\run}}{\curr - \sol}|\mathcal{F}_{\run} \right] \right]
	\notag\\
	&=\sum_{\run=\start}^{\nRuns}\ex\left[\step_{\run}\ex\left[\braket{\noise_{\run}}{\curr - \sol}|\mathcal{F}_{\run} \right] \right]
	\notag\\
	&=\sum_{\run=\start}^{\nRuns}\ex\left[ \step_{\run}\braket{\ex[\noise_{\run}|\mathcal{F}_{\run}]}{\state_{\run}-\sol}\right]
	=0
\end{align}
with the third and the fourth equality being obtained by the fact that $\step_{\run}$ and $\state_{\run}$ are $\mathcal{F}_{\run}-$ measurable.
\item
For the term $\ex\left[ \sum_{\run=\start}^{\nRuns}\step_{\run}^{2}\res_{\run}^{2}\right]$:
By applying \cref{lem:logarithmic-3}  and \cref{lem:bound-stoch}, we have with probability $1$:
\begin{align}
\sum_{\run=1}^{\nRuns}\step_{\run}^{2}\res_{\run}^{2}
	\leq 2+ \frac{4\tilde{\gbound}^{2}}{\res_{0}^{2}}+2\log(1+\sum_{\run=1}^{\nRuns}\frac{\res_{\run}^{2}}{\res_{0}^{2}})
	\leq 2 + \frac{4\tilde{\gbound}^{2}}{\res_{0}^{2}}+2\log(1+\frac{\tilde{\gbound}^{2}}{\hstr\res_{0}^{2}}\nRuns)
\end{align}
Therefore we get:
\begin{equation}
\ex\left[ \sum_{\run=\start}^{\nRuns}\step_{\run}^{2}\res_{\run}^{2}\right]\leq 2 + \frac{4\tilde{\gbound}^{2}}{\res_{0}^{2}}+2\log(1+\frac{\tilde{\gbound}^{2}}{\res_{0}^{2}}\nRuns)
\end{equation}
\end{itemize}
Thus, combining all the above we obtain:
\begin{equation}
\frac{1}{\sqrt{(\res_{0}^{2}+\tilde{\gbound}^{2})\nRuns}} \ex\left[\sum_{\run=1}^{\nRuns}(\obj(\state_{\run})-\obj(\sol))\right]\leq \breg(\sol,\state_1)+2 + \frac{4\tilde{\gbound}^{2}}{\res_{0}^{2}}
+2\log(1+\frac{\tilde{\gbound}^{2}}{\res_{0}^{2}}\nRuns)
\end{equation}
and hence,
\begin{equation}
 \ex\left[\sum_{\run=1}^{\nRuns}(\obj(\state_{\run})-\obj(\sol))\right]\leq \sqrt{(\res_{0}^{2}+\tilde{\gbound}^{2})\nRuns}\left[\breg(\sol,\state_1)+2 + \frac{4\tilde{\gbound}^{2}}{\res_{0}^{2}}
+2\log(1+\frac{\tilde{\gbound}^{2}}{\res_{0}^{2}})\nRuns)\right]
 \end{equation}
 The result follows by dividing both sides by $\nRuns$.
\end{proof}

\begin{proof}[Proof of \cref{thm:stoch-smooth}]
By \cref{prop:RS-prop}, we have:
\begin{align}
\obj(\state_{\run+1})&\leq \obj(\state_{\run})+\braket{\nabla \obj(\state_{\run})}{\state_{\run+1}-\state_{\run}}+\smooth \breg(\state_{\run+1},\state_{\run})
	\notag\\
	&\leq \obj(\state_{\run})+\braket{\nabla \obj(\state_{\run})}{\state_{\run+1}-\state_{\run}}+\smooth \left[\breg(\state_{\run+1},\state_{\run})+\breg(\state_{\run},\state_{\run+1})\right]
	\notag\\
	&=\obj(\state_{\run})+\braket{\signal_{\run}}{\state_{\run+1}-\state_{\run}}+\braket{\noise_{\run}}{\state_{\run}-\state_{\run+1}}+\smooth \step_{\run}^{2}\res_{\run}^{2}
	\notag\\
	&\leq \obj(\state_{\run})-\frac{1}{\step_{\run}}\left [\breg(\state_{\run+1},\state_{\run})+\breg(\state_{\run},\state_{\run+1}) \right]+\dnorm{\noise_{\run}}\norm{\state_{\run}-\state_{\run+1}}+\smooth \step_{\run}^{2}\res_{\run}^{2}
	\notag\\
	&= \obj(\state_{\run})-\step_{\run}\res_{\run}^{2}+\dnorm{\noise_{\run}}\norm{\state_{\run}-\state_{\run+1}}+\smooth \step_{\run}^{2}\res_{\run}^{2}
\end{align}
Now, since $\hreg$ is $\hstr-$ strongly convex we have that:
\begin{equation}
\norm{\state_{\run}-\state_{\run+1}}\leq \sqrt{\frac{2}{\hstr}\left[\breg(\state_{\run+1},\state_{\run})+\breg(\state_{\run},\state_{\run+1}) \right]}=\sqrt{\frac{2}{\hstr}}\step_{\run}\res_{\run}
\end{equation}
and using the fact that the noise $\dnorm{\noise_{\run}}\leq \sdev$ almost surely, we have:
\begin{equation}
\obj(\state_{\run+1})\leq \obj(\state_{\run})-\step_{\run}\res_{\run}^{2}+\sqrt{\frac{2}{\hstr}}\step_{\run}\res_{\run}^{2}+\smooth\step_{\run}^{2}\res_{\run}^{2}
\end{equation}
Therefore, after rearranging and telescoping we get:
\begin{equation}
\sum_{\run=1}^{\nRuns}\step_{\run}\res_{\run}^{2}\leq 2\left[\obj(\state_1)-\min_{\point \in \points}\obj(\point)+\sum_{\run=1}^{\nRuns}\step_{\run}\res_{\run}^{2}(\smooth \step_{\run}-\frac{1}{2})+\sdev \sqrt{\frac{2}{\hstr}}\sum_{\run=1}^{\nRuns}\step_{\run}\res_{\run}\right]
\end{equation}
Now, let us bound each term of the (RHS) of the above individually:
\begin{itemize}
\item
For the term $\sum_{\run=1}^{\nRuns}\step_{\run}\res_{\run}^{2}(\smooth \step_{\run}-\frac{1}{2})$ we first set:
\begin{equation}
\run_{0}=\max \{1\leq \run \leq \nRuns: \step_{\run}\geq \frac{1}{2\smooth}\}
\end{equation}
Then, by decomposing the  said sum we get:
\begin{align}
\sum_{\run=1}^{\nRuns}\step_{\run}\res_{\run}^{2}(\smooth \step_{\run}-\frac{1}{2})
	&=\sum_{\run=1}^{\run_0}\step_{\run}\res_{\run}^{2}(\smooth \step_{\run}-\frac{1}{2})+\sum_{\run=\run_{0}+1}^{\nRuns}\step_{\run}\res_{\run}^{2}(\smooth \step_{\run}-\frac{1}{2})
	\notag\\
	&\leq \sum_{\run=1}^{\run_0}\step_{\run}\res_{\run}^{2}(\smooth \step_{\run}-\frac{1}{2})
	\notag\\
	&\leq \smooth \sum_{\run=1}^{\run_0}\step_{\run}^{2}\res_{\run}^{2}
\end{align}
with the second inequality being obtained by the definition of $\run_{0}$.
Now, due to the fact that $\res_{\run}^{2}\leq \tilde{\gbound}^{2}$ almost surely (by invoking \cref{lem:bound-stoch}) we have:
\begin{align}
\smooth\sum_{\run=1}^{\run_0}\step_{\run}^{2}\res_{\run}^{2}
	&=\smooth \sum_{\run=1}^{\run_0}\frac{\res_{\run}^{2}}{\res_{0}^{2}+\sum_{\runalt=1}^{\run-1}\res_{\runalt}^{2}}
	\notag\\
	&\leq \smooth \left[2+\frac{4\tilde{\gbound}^{2}}{\res_{0}^{2}}+2\log(1+\frac{1}{\res_{0}^{2}}\sum_{\run=1}^{\run_{0}-1}\res_{\run}^{2}) \right]
	\notag\\
	&\leq \smooth \left[ 2+\frac{4\tilde{\gbound}^{2}}{\res_{0}^{2}}+2\log\frac{1}{\res_{0}^{2}}(\res_{0}^{2}+\sum_{\run=1}^{\run_0-1}\res_{\run}^{2})\right]
	\notag\\
	&\leq \smooth \left[ 2+\frac{4\tilde{\gbound}^{2}}{\res_{0}^{2}}+2\log \frac{1}{\res_{0}^{2}\step_{\run_0}^{2}}\right]
\end{align}
Therefore, by the definition of $\run_0$ we finally get with probability $1$:
\begin{equation}
\sum_{\run=1}^{\nRuns}\step_{\run}\res_{\run}^{2}(\smooth \step_{\run}-\frac{1}{2})\leq \smooth\left[ 2+\frac{4\tilde{\gbound}^{2}}{\res_{0}^{2}}+2\log\frac{4\smooth^{2}}{\res_{0}^{2}}\right]
\end{equation}
\item
For the term $\sdev \sqrt{\frac{2}{\hstr}}\sum_{\run=1}^{\nRuns}\step_{\run}\res_{\run}$ we have:
\begin{align}
\sdev \sqrt{\frac{2}{\hstr}}\sum_{\run=1}^{\nRuns}\step_{\run}\res_{\run}&=\sdev \sqrt{\frac{2}{\hstr}}\sum_{\run=1}^{\nRuns}\sqrt{\step_{\run}^{2}\res_{\run}^{2}}
	\leq \sdev \sqrt{\frac{2}{\hstr}}\sqrt{\nRuns}\sqrt{\sum_{\run=1}^{\nRuns}\step_{\run}^{2}\res_{\run}^{2}}
\end{align}
Therefore, by working in the same spirit as above we get:
\begin{align}
\sdev \sqrt{\frac{2}{\hstr}}\sum_{\run=1}^{\nRuns}\step_{\run}\res_{\run}
	&\leq \sdev \sqrt{\frac{2}{\hstr}}\sqrt{2+\frac{4\tilde{\gbound}^{2}}{\res_{0}^{2}}+2\log(1+\frac{1}{\res_{0}^{2}}\sum_{\run=1}^{\nRuns}\res_{\run}^{2})}
	\notag\\
	&\leq \sdev \sqrt{\frac{2}{\hstr}}\sqrt{\nRuns}\sqrt{2+\frac{4\tilde{\gbound}^{2}}{\res_{0}^{2}}+2\log(1+\frac{\tilde{\gbound}^{2}}{\res_{0}^{2}}\nRuns)}
\end{align}
\end{itemize} 
On the other hand, we may the (LHS) from below as follows:
\begin{align}
\sum_{\run=1}^{\nRuns}\step_{\run}\res_{\run}^{2}
	\geq \step_{\nRuns}\sum_{\run=1}^{\nRuns}\res_{\run}^{2}
	\geq \step_{\nRuns}\left[\res_{0}^{2} -\res_{0}^{2}+ \sum_{\run=1}^{\nRuns}\res_{\run}^{2}\right]
	=\frac{\step_{\nRuns}}{\step_{\nRuns+1}^{2}}-\res_{0}^{2}\step_{\nRuns}
	=\frac{1}{\step_{\nRuns}}-\res_{0}^{2}\step_{\nRuns}
\end{align}
So, combining the above:
\begin{multline}
\frac{1}{\step_{\nRuns}}-\res_{0}^{2}\step_{\nRuns}
	\leq 2(\obj(\state_1)-\min_{\point \in \points} \obj(\point)+
\smooth\left[2+\frac{4\tilde{\gbound}^{2}}{\res_{0}^{2}}+2\log\frac{4\smooth^{2}}{\res_{0}^{2}}\right]
\\
+\sdev \sqrt{\frac{2}{\hstr}}\sqrt{\nRuns}\sqrt{2+\frac{4\tilde{\gbound}^{2}}{\res_{0}^{2}}+2\log(1+\frac{\tilde{\gbound}^{2}}{\res_{0}^{2}}\nRuns)})
\end{multline}
which finally yields with probability $1$:
\begin{equation}
\frac{1}{\step_{\nRuns}}\leq \res_{0}+2(\obj(\state_1)-\min_{\point \in \points} \obj(\point)+
\smooth\left[2+\frac{4\tilde{\gbound}^{2}}{\res_{0}^{2}}+2\log\frac{4\smooth^{2}}{\res_{0}^{2}}\right]
+\sdev \sqrt{\frac{2}{\hstr}}\sqrt{\nRuns}\sqrt{2+\frac{4\tilde{\gbound}^{2}}{\res_{0}^{2}}+2\log(1+\frac{\tilde{\gbound}^{2}}{\res_{0}^{2}}\nRuns)})
\end{equation}
and hence with probability $1$:
\begin{equation*}
\step_{\nRuns}
	\geq \bracks*{
		\res_{0}+2(\obj(\state_1)-\min_{\point \in \points} \obj(\point)
		+ \smooth\left[2+\frac{4\tilde{\gbound}^{2}}{\res_{0}^{2}}+2\log\frac{4\smooth^{2}}{\res_{0}^{2}}\right]
		+\sdev \sqrt{\frac{2}{\hstr}}\sqrt{\nRuns}\sqrt{2+\frac{4\tilde{\gbound}^{2}}{\res_{0}^{2}}+2\log(1+\frac{\tilde{\gbound}^{2}}{\res_{0}^{2}}\nRuns)})
		}^{-1}
\qedhere
\end{equation*}
Therefore, by setting:
\begin{equation}
A=\res_{0}+2(\obj(\state_1)-\min_{\point \in \points} \obj(\point)
		+ \smooth\left[2+\frac{4\tilde{\gbound}^{2}}{\res_{0}^{2}}+2\log\frac{4\smooth^{2}}{\res_{0}^{2}}\right]
\end{equation}
and 
\begin{equation}
B=\sdev \sqrt{\frac{2}{\hstr}}\sqrt{2+\frac{4\tilde{\gbound}^{2}}{\res_{0}^{2}}+2\log(1+\frac{\tilde{\gbound}^{2}}{\res_{0}^{2}}\nRuns)})
\end{equation}
we get that:
\begin{equation}
\ex\left[\sum_{\run=1}^{\nRuns}(\obj(\state_{\run})-\obj(\sol))\step_{\nRuns} \right]\geq \left(A+B\sqrt{\nRuns}\right)^{-1}\ex\left[\sum_{\run=1}^{\nRuns}(\obj(\state_{\run})-\obj(\sol)) \right]
\end{equation}
Moreover, working in the same spirit as in \cref{thm:stoch} we have:
\begin{equation}
 \left(A+B\sqrt{\nRuns}\right)^{-1}\ex\left[\sum_{\run=1}^{\nRuns}(\obj(\state_{\run})-\obj(\sol)) \right]
\leq\ex\left[\sum_{\run=1}^{\nRuns}(\obj(\state_{\run})-\obj(\sol))\step_{\nRuns} \right]\leq\left(D_{1}+\ex\left[ \sum_{\run=1}^{\nRuns}\step_{\run}^{2}\res_{\run}^{2}\right] \right) 
\end{equation}
which in turn yields:
\begin{equation}
\ex\left[\sum_{\run=1}^{\nRuns}(\obj(\state_{\run})-\obj(\sol)) \right]\leq \left(D_{1}+\ex\left[ \sum_{\run=1}^{\nRuns}\step_{\run}^{2}\res_{\run}^{2}\right] \right)  \left(A+B\sqrt{\nRuns}\right)
\end{equation}
The result then follows by dividing both sides by $\nRuns$ and by the fact that $\ex\left[ \sum_{\run=1}^{\nRuns}\step_{\run}^{2}\res_{\run}^{2}\right]=\bigoh(\log\nRuns)$. 
\end{proof}

\section{Last iterate Convergence}
\label{app:Last}
%
%
 Throughout this section we assume that $\obj$ satisfies the following weak-secant inequality of the  form:
\begin{equation}
\tag{SI}
\label{eq:secant}
\inf\setdef{\braket{\nabla\obj(\point)}{\point - \sol}}{\sol\in\argmin\obj,\point\in\cpt}
	> 0
\end{equation}
for every closed subset $\cpt$ of $\points$ that is separated by neighborhoods from $\argmin\obj$. More precisely, our proof is divided in two parts. To begin with, we first show that under \eqref{eq:RC} or \eqref{eq:RS} the iterates of\method possess a convergent subsequence towards the solution set $\sols$.
 Formally stated, we have the following proposition:
 
\begin{proposition}
\label{prop:conv-subsequence1}
Assume that $\obj$ is \eqref{eq:RC} or \eqref{eq:RS} and $\state_{\run}$ are the iterates generated by\method. Then, there exists a subsequence $\state_{k_{\run}}$ which converges to the solution set $\sols$. 
\end{proposition}

\begin{proof}
Assume to the contrary that the sequence $\state_{\run}$ generated by\method admits no limit points in $\sols=\argmin \obj$. Then, there exists a (non-empty) closed set $\cpt \subseteq \points$ which is separated by neighbourhoods from $\argmin \obj$ and is such that $\state_{\run}\in \cvx$ for all sufficiently large $\run$. Then, by relabelling $\state_{\run}$ if necessary, we can assume without loss of generality that $\state_{\run} \in \cpt$ for all $\run \in \N$. Thus, following the spirit of \cref{lem:bound-2}, we have:

\begin{align}
\breg(\sol,\state_{\run+1})
	&\leq \breg(\sol,\state_{\run})-\step_{\run}\braket{\nabla \obj(\state_{\run})}{\state_{\run}-\sol}+\breg(\state_{\run},\state_{\run+1})
	\notag\\
	&\leq \breg(\sol,\state_{\run})-\step_{\run}\braket{\nabla \obj(\state_{\run})}{\state_{\run}-\sol}+\left[\breg(\state_{\run},\state_{\run+1})+\breg(\state_{\run+1},\state_{\run})\right]
	\notag\\
	&= \breg(\sol,\state_{\run})-\step_{\run}\braket{\nabla \obj(\state_{\run})}{\state_{\run}-\sol}+\step_{\run}^{2}\res_{\run}^{2}
\end{align}
with the last equality being obtained by the definition of  \eqref{eq:res-Breg}. Now, applying \eqref{eq:secant}
we get:
\begin{equation}
\breg(\sol,\state_{\run+1})\leq \breg(\sol,\state_{\run})-\step_{\run}\delta(\cpt)+\step_{\run}^{2}\aux_{\run}^{2}
\end{equation}
with $\delta(\cpt)=\inf \{\braket{\nabla \obj(\point)}{\point-\sol}: \sol \in \argmin \obj, \point \in \cpt\}>0$. Hence, 
by telescoping $\run=1,\dotsc, \nRuns$, factorizing and setting $\beta_{\run}=\sum_{\run=1}^{\nRuns}\step_{\run}$ we have:
\begin{equation}
\label{eq:subseq}
\breg(\sol,\state_{\nRuns+1})\leq \breg(\sol,\state_{1})-\beta_{\run}\left[ \delta(\cpt)-\frac{\sum_{\run=1}^{\nRuns}\step_{\run}^{2}\aux_{\run}^{2}}{\beta_{\run}}\right]
\end{equation}
Now, \eqref{eq:subseq} will be the crucial lemma that will walk throughout our analysis. In particular, we will treat the different regularity conditions of \eqref{eq:RC} and \eqref{eq:RS} seperately.

\para{Case 1: The \eqref{eq:RC} case}
Assume that $\obj$ satisfies \eqref{eq:RC}. By examining the asymptotic behaviour of each term individually, we obtain:
\begin{itemize}
\item
For the term $\beta_{\nRuns}=\sum_{\run=1}^{\nRuns}\step_{\run}$, we have:
\begin{equation}
\beta_{\nRuns}=\sum_{\run=1}^{\nRuns}\frac{1}{\sqrt{\res_{0}^{2}+\sum_{j=1}^{\run-1}\res_{\run}^{2}}}\geq \sum_{\run=1}^{\nRuns}\frac{1}{\sqrt{\res_{0}^{2}+2\gbound^{2}\run}}
\end{equation}
which yields that $\beta_{\nRuns}\to +\infty$ and more precisely $\beta_{\nRuns}=\Omega(\sqrt{\nRuns})$.
\item
For the term $\frac{\sum_{\run=1}^{\nRuns}\step_{\run}^{2}\res_{\run}^{2}}{\beta_{\nRuns}}$, for the numerator we have:
\begin{align}
\sum_{\run=1}^{\nRuns}\step_{\run}^{2}\res_{\run}^{2}
	&=\sum_{\run=1}^{\nRuns}\frac{\res_{\run}^{2}}{\res_{0}^{2}+\sum_{j=1}^{\run-1}\res_{j}^{2}/\res_{0}^{2}}
	\notag\\
	&\leq 2+8\gbound^{2}/\res_{0}^{2}+2\log(1+\sum_{\run=1}^{\nRuns-1}\res_{\run}^{2}/\res_{0}^{2})
	\notag\\
	&\leq 2+8\gbound^{2}/\res_{0}^{2}+2\log(1+2\gbound^{2}\nRuns/\res_{0}^{2})
\end{align}
which yields that $\sum_{\run=1}^{\nRuns}\step_{\run}^{2}\res_{\run}^{2}=\bigoh(\log\nRuns)$, and combined with the fact that $\beta_{\run}=\Omega(\sqrt{\nRuns})$ we readily get: 
\begin{equation}
\frac{\sum_{\run=1}^{\nRuns}\step_{\run}^{2}\res_{\run}^{2}}{\beta_{\nRuns}}\to 0
\end{equation}
\end{itemize}
So, combining all the above and letting $\nRuns \to +\infty$ in \eqref{eq:subseq},  we get that $\breg(\sol,\state_{\nRuns+1})\to -\infty$, a contradiction. Therefore, the result under \eqref{eq:RC} follows.

\para{Case 2:  The \eqref{eq:RS} case}
On the other hand, assume that $\obj$ satisfies \eqref{eq:RS}. Recalling \cref{lem:smooth-Breg}
and the fact that $\step_{\run}$ is decreasing we have:
\begin{equation}
\sum_{\run=1}^{\nRuns}\step_{\run}\res_{\run}^{2}\leq \sum_{\run=1}^{+\infty}\res_{\run}^{2}<+\infty
\end{equation}
which by working as in \cref{lem:smooth-Breg} also yields:
\begin{equation}
\lim_{\run \to +\infty}\step_{\run}=\gamma_{\infty}>0
\end{equation}
Additionally, since $\step_{\run}$ is decreasing and bounded we also have that $\gamma_{\infty}=\inf_{\run}\step_{\run}$.
Now, we shall re-examine the terms of \eqref{eq:subseq}. More precisely, we have:
\begin{itemize}
\item
For $\beta_{\nRuns}$ we have:
\begin{equation}
\label{eq:linear}
\beta_{\nRuns}=\sum_{\run=1}^{\nRuns}\step_{\run}\geq \gamma_{\infty}\sum_{\run=1}^{\nRuns}1=\gamma_{\infty}\nRuns
\end{equation}
which in turn yields that $\beta_{\nRuns}\to +\infty$ and more precisely $\beta_{\nRuns}=\Omega(\nRuns)$.
\item
For the term $\frac{\sum_{\run=1}^{\nRuns}\step_{\run}^{2}\res_{\run}^{2}}{\beta_{\nRuns}}$, for the numerator we have by the fact that $\step_{\run}\leq 1/\res_{0}$ and \cref{lem:smooth-Breg}:
\begin{equation}
\sum_{\run=1}^{\nRuns}\step_{\run}\res_{\run}^{2}\leq \frac{1}{\res_{0}}\sum_{\run=1}^{\nRuns}\res_{\run}^{2}<+\infty
\end{equation}
which yields that $\sum_{\run=1}^{\nRuns}\step_{\run}^{2}\res_{\run}^{2}=\bigoh(1)$, which combined with
\eqref{eq:linear} gives that:
\begin{equation}
\frac{\sum_{\run=1}^{\nRuns}\step_{\run}^{2}\res_{\run}^{2}}{\beta_{\nRuns}}\to 0
\end{equation}
\end{itemize}
so, again combing the above and letting $\nRuns \to +\infty$ in \eqref{eq:subseq}, we get that $\breg(\sol,\state_{\nRuns+1})\to -\infty$, a contradiction.
Therefore, the result follows also under \eqref{eq:RS}.
\end{proof}

Having all this at hand, we are finally in the position to prove the convergence of the actual iterates of the method.  For that we will need an intermediate lemma that  shall allow us to pass from a convergent subsequence to global convergence (see also \cite{Com01}, \cite{Pol87}). 

\begin{lemma}
\label{lem:quasi-Fejer}
Let $\chi \in (0,1]$, $(\alpha_{\run})_{\run \in \N}$, $(\beta_{\run})_{\run \in \N}$ non-negative sequences and
$(\varepsilon_{\run})_{\run \in \N}\in l^{1}(\N)$ such that $\run=1,2,\dotsc$:
\begin{equation}
\alpha_{\run+1}\leq \chi \alpha_{\run}-\beta_{\run}+\varepsilon_{\run}
\end{equation}
Then, $\alpha_{\run}$ converges.
\end{lemma}

\begin{proof}
First, one shows that $\alpha_{\run \in \N}$ is a bounded sequence.  Indeed, one can derive directly that:
\begin{equation}
\alpha_{\run+1}\leq \chi^{\run+1}\alpha_0+\sum_{k=0}^{\run}\chi^{\run-k}\varepsilon_k
\end{equation}
Hence, $(\alpha_{\run})_{\run \in \N}$ lies in $[0,\alpha_0+\varepsilon]$, with $\varepsilon=\sum_{\run=0}^{+\infty}\varepsilon_{\run}$. Now, one is able to extract a convergent subsequence $(\alpha_{k_{\run}})_{\run \in \N}$, let say $\lim_{\run \to +\infty}\alpha_{k_{\run}}=\alpha \in [0,\alpha_0+\varepsilon]$ and fix $\delta >0$. Then, one can find some $\run_0$ such that $\alpha_{k_{\run_0}}-\alpha < \frac{\delta}{2}$ and $\sum_{m>\run_{k_{\run_0}}}\varepsilon_m<\frac{\delta}{2}$. That said, we have:
\begin{equation}
0\leq \alpha_{\run}\leq \alpha_{k_{\run_0}}+\sum_{m>\run_{k_{\run_0}}}\varepsilon_m<\frac{\delta}{2}+\alpha+\frac{\delta}{2}=\alpha+\delta
\end{equation}
Hence, $\limsup_{\run}\alpha_{\run}\leq \liminf_{\run}\alpha_{\run}+\delta$. Since, $\delta$ is chosen arbitrarily the result follows.
\end{proof}

\begin{proof}[Proof of \cref{thm:last-iterate}]

We will divide our proof in two parts by distinguishing the two different regularity cases.
\para{Case 1: The \eqref{eq:RC} case}
Given that $\step_{\run}$ is decreasing and bounded from below we have that its limit exists, denoted by $\gamma_{\infty}\geq 0$.
We shall consider two cases:
\begin{enumerate}
\item
\emph{$\gamma_{\infty}>0$}: Following the same reasoning with \cref{lem:smooth-Breg} we get that:
\begin{equation}
\sum_{\run=1}^{\nRuns}\step_{\run}^{2}\res_{\run}^{2}\leq \sum_{\run=1}^{+\infty}\res_{\run}^{2}<+\infty
\end{equation}
Hence, by recalling the inequality:
\begin{equation}
\breg(\sol,\state_{\run+1})\leq \breg(\sol,\state_{\run})+\step_{\run}^{2}\res_{\run}^{2}\;\;\text{for all}\;\;\sol \in \sols
\end{equation}
whereas after taking infima on both sides with respect to $\sols$, we get:
\begin{equation}
\inf_{\sol\in \sols}\breg(\sol,\state_{\run+1})\leq \inf_{\sol \in \sols}\breg(\sol,\state_{\run})+\step_{\run}^{2}\res_{\run}^{2}
\end{equation}
and since the sequence $\step_{\run}^{2}\res_{\run}^{2}$ is summable we can directly apply \cref{lem:quasi-Fejer} which yields that the sequence $\inf_{\sol \in \sols}\breg(\sol,\state_{\run})$ is convergent.  Now, since by \cref{prop:conv-subsequence1},\method possesses a convergent subsequence towards the solution set $\sols$ the result follows.
\item
\emph{$\gamma_{\infty}=0$}: Pick some $\eps >0$ and consider the Bregman zone:
\begin{equation}
D_{\eps}=\{\point \in \points : \breg(\sols,\point)<\eps\}. 
\end{equation}
Then, it suffices to show that $\state_{\run}\in D_{\eps}$ for all sufficiently large $\run$. In doing so, consider the inequality:
\begin{align}
\breg(\sol,\state_{\run+1})
	&\leq \breg(\sol,\state_{\run})-\step_{\run}\braket{\nabla \obj(\state_{\run})}{\state_{\run}-\sol}+\step_{\run}^{2}\res_{\run}^{2}
	\notag\\
	&\leq \breg(\sol,\state_{\run})-\step_{\run}\braket{\nabla \obj(\state_{\run})}{\state_{\run}-\sol}+\step_{\run}^{2}
\frac{2\gbound^{2}}{\hstr}
\end{align}
with the second inequality being obtained by \cref{lem:bounded-Breg}. To proceed, assume inductively that $\state_{\run}\in D_{\eps}$. By the regularity assumptions of the regularizer $\hreg$, it follows that there exists a $\delta-$ neighbourhood contained in the closure of $D_{\eps/2}$. So, by the \eqref{eq:secant} condition we have:
\begin{equation}
\braket{\obj(\point)}{\point-\sol}\geq c>0\;\;\text{for some $c\equiv c(\eps)>0$ and for all $\point \in D_{\eps}\setminus D_{\eps/2}$ and $\sol \in \sols$}
\end{equation}
We consider two cases:
\begin{itemize}
\item
$\state_{\run}\in D_{\eps}\setminus D_{\eps/2}$: In. this case,  we have:
\begin{align}
\breg(\sol,\state_{\run+1})
	&\leq \breg(\sol,\state_{\run})-\step_{\run}\braket{\nabla \obj(\state_{\run})}{\state_{\run}-\sol}+\step_{\run}^{2}\frac{2\gbound^{2}}{\hstr}
	\notag\\
	&\leq \breg(\sol,\state_{\run})-\step_{\run}c+\step_{\run}^{2}\frac{2\gbound^{2}}{\hstr}
\end{align}
Thus, provided that $\step_{\run}\leq \frac{c\hstr}{2\gbound^{2}}$ we get that $\breg(\sol,\state_{\run+1})\leq \breg(\sol,\state_{\run})$. Hence, by taking infima on both sides relative to $\sol \in \sols$, we get that
$\breg(\sols,\state_{\run+1})\leq \breg(\sols,\state_{\run})<\eps$.
\item
$\state_{\run} \in D_{\eps/2}$: In this case, we have:
\begin{align}
\breg(\sol,\state_{\run+1})
	&\leq \breg(\sol,\state_{\run})-\step_{\run}\braket{\nabla \obj(\state_{\run})}{\state_{\run}-\sol}+\step_{\run}^{2}\frac{2\gbound^{2}}{\hstr}
	\notag\\
	&\leq \breg(\sol,\state_{\run})+\step_{\run}^{2}\frac{2\gbound^{2}}{\hstr}
\end{align}
with the second inequality being obtained by the optimality of $\sol$. Now, provided that $\step_{\run}^{2}\leq \frac{\eps \hstr}{4\gbound^{2}}$ or equivalently $\step_{\run}\leq \frac{\sqrt{\eps \hstr}}{2\gbound}$ we have:
\begin{equation}
\breg(\sol,\state_{\run+1})\leq \breg(\sol,\state_{\run})+\frac{\eps}{2}
\end{equation}
whereas again by taking infima on both sides we get that $\breg(\sols,\state_{\run+1})\leq \breg(\sols,\state_{\run})+\frac{\eps}{2}<\eps$.
\end{itemize}
Hence, summarizing we have that $\state_{\run+1}\in D_{\eps}$ whenever $\state_{\run}\in D_{\eps}$ and
$\step_{\run}\leq \min\{\frac{c\hstr}{2\gbound^{2}},\frac{\sqrt{\eps\hstr}}{2\gbound}\}$. Hence, the result follows by. \cref{prop:conv-subsequence1} and the fact that $\step_{\run}\to 0$.
\end{enumerate}

\para{Case 2: The \eqref{eq:RS} case}
Recall that we have the following inequality, 
\begin{equation}
\breg(\sol,\state_{\run+1})\leq \breg(\sol,\state_{\run})+\step_{\run}^{2}\res_{\run}^{2}\;\;\text{for all}\;\;\sol \in \sols
\end{equation}
whereas taking infima on both sides relative to $\sols$ we readily get:
\begin{equation}
\inf_{\sol \in \sols}\breg(\sol,\state_{\run+1})\leq \inf_{\sol \in \sols}\breg(\sol,\state_{\run})+\step_{\run}^{2}\res_{\run}^{2}
\end{equation}
Now, by recalling that by \cref{lem:smooth-Breg}, we have $\step_{\run}^{2}\res_{\run}^{2}$ is summable. we can apply directly \cref{lem:quasi-Fejer}. Thus, we have the sequence $\inf_{\sol \in \sols}\breg(\sol,\state_{\run})$ is convergent.
Moreover, \cref{prop:conv-subsequence1} guarantees that there a subsequence of $\inf_{\sol \in \sols}\norm{\state-\sol}^{2}$ that converges to $0$. We obtain that there exists also a subsequence of $\inf_{\sol \in \sols}\breg(\sol,\state_{\run})$  that converges to $0$ and since $\inf_{\sol \in \sols}\breg(\sol,\state_{\run})$ is convergent, we readily get that:
\begin{equation}
\inf_{\sol\in \sols}\norm{\sol-\state_{\run}}^{2}\leq \inf_{\sol \in \sols}\breg(\sol,\state_{\run})\to 0
\end{equation}
and the proof is complete.
\end{proof}

\section{Lemmas on numerical sequences}
\label{app:sequences}

In this appendix, we provide some necessary inequalities on numerical sequences that we require for the convergence rate analysis of the previous sections.
Most of the lemmas presented below already exist in the literature, and go as far back as \citet{ACBG02} and \citet{MS10};
when appropriate, we note next to each lemma the references with the statement closest to the precise version we are using in our analysis.
These lemmas can also be proved by the general methodology outlined in \citet[Lem.~14]{GSE14}, so we only provide a proof for two ancillary results that would otherwise require some more menial bookkeeping.

\smallskip


\begin{lemma}[\citealp{MS10}, \citealp{LYC18}]
\label{lem:sqrt}
For all non-negative numbers $\alpha_{1},\dotsc \alpha_{\run}$, the following inequality holds:
\begin{equation}
\sqrt{\sum_{\run=1}^{\nRuns}\alpha_{\run}}\leq \sum_{\run=1}^{\nRuns}\dfrac{\alpha_{\run}}{\sqrt{\sum_{i=1}^{\run}\alpha_{i}}}\leq 2\sqrt{\sum_{\run=1}^{\nRuns}\alpha_{\run}}
\end{equation}
\end{lemma}

\begin{lemma}[\citealp{LYC18}]
\label{lem:logarithmic-1}
For all non-negative numbers $\alpha_{1},\dotsc \alpha_{\run}$, the following inequality holds:
\begin{equation}
\sum_{\run=1}^{\nRuns}\dfrac{\alpha_{\run}}{1+\sum_{i=1}^{\run}\alpha_{i}}\leq 1+\log(1+\sum_{\run=1}^{\nRuns}\alpha_{\run})
\end{equation}
\end{lemma}


\begin{lemma}
\label{lem:logarithmic-2}
Let $b_{1},\dotsc, b_{\run}$ a sequence of non-negative numbers with $b_{1}>0$. Then, the following inequality holds:
\begin{equation}
\sum_{\run=1}^{\nRuns}
	\dfrac{b_{\run}}{\sum_{i=1}^{\run}b_{i}}
	\leq 2+\log\parens*{\dfrac{\sum_{\run=1}^{\nRuns}b_{\run}}{b_{1}}}
\end{equation}
\end{lemma}
\begin{proof}
It is directly obtained by applying \cref{lem:logarithmic-1} for the sequence $\alpha_{\run}=b_{\run}/b_{1}$.
\end{proof}

The following set of inequalities are due to \cite{BL19}. For completeness, we provide a sketch of their proof.

\begin{lemma}[\citealp{BL19}]
\label{lem:ineq}
For all non-negative numbers: $\alpha_{1},\dotsc \alpha_{\run}\in [0,\alpha]$, $\alpha_{0}\geq 0$, the following inequality holds:
\begin{equation}
\sqrt{\alpha_{0}+\sum_{\run=1}^{\nRuns-1}\alpha_{i}}-\sqrt{\alpha_{0}}\leq \sum_{\run=1}^{\nRuns}\frac{\alpha_{\run}}{\sqrt{\alpha_{0}+\sum_{i=1}^{\run-1}\alpha_{j}}}\leq \frac{2\alpha}{\sqrt{\alpha_{0}}}+3\sqrt{\alpha}+3\sqrt{\alpha_{0}+\sum_{\run=1}^{\nRuns-1}\alpha_{\run}}
\end{equation}
\end{lemma}

\begin{lemma}
\label{lem:logarithmic-3}
For all non-negative numbers: $\alpha_{1},\dotsc \alpha_{\run}\in [0,\alpha]$, $\alpha_{0}\geq 0$, we have:
\begin{equation}
\sum_{\run=1}^{\nRuns}\frac{\alpha_{\run}}{\alpha_{0}+\sum_{i=1}^{\run-1}\alpha_{i}}
	\leq 2
		+ \frac{4\alpha}{\alpha_{0}}
		+ 2\log\parens*{1+\sum_{\run=1}^{\nRuns-1}\frac{\alpha_{\run}}{\alpha_{0}}}
\end{equation}
\end{lemma}

\begin{proof}
Let us denote
\begin{equation}
\txs
\nRuns_{0}
	= \min \setdef[\big]{\run \in [\nRuns]}{\sum_{j=1}^{\run-1}\alpha_{j}\geq \alpha}
\end{equation}
Then, dividing the sum by $\nRuns_{0}$, we get:
\begin{align}
\sum_{\run=1}^{\nRuns}\frac{\alpha_{\run}}{\alpha_{0}+\sum_{i=1}^{\run-1}\alpha_{i}}&
	\leq \sum_{\run=1}^{\nRuns_{0}-1}\frac{\alpha_{\run}}{\alpha_{0}+\sum_{i=1}^{\run-1}\alpha_{i}}  
	+\sum_{\run=\nRuns_{0}}^{\nRuns}\frac{\alpha_{\run}}{\alpha_{0}+\sum_{i=1}^{\run-1}\alpha_{i}}
	\notag\\
	&\leq \frac{1}{\alpha_{0}}\sum_{\run=1}^{\nRuns_{0}-1}\alpha_{\run}+\sum_{\run=\nRuns_{0}}^{\nRuns}
	\dfrac{\alpha_{\run}}{1/2\alpha_{0}+1/2\alpha+1/2\sum_{j=1}^{\run-1}\alpha_{j}}
	\notag\\
	&\leq \frac{\alpha}{\alpha_{0}}+2\sum_{\run=\nRuns_{0}}^{\nRuns}\dfrac{\alpha_{i}/\alpha_{0}}{1+\sum_{j=\nRuns_{0}}^{\run}\alpha_{j}/\alpha_{0}}
	\notag\\
	&\leq \frac{2\alpha}{\alpha_{0}}
		+ 2
		+ 2\log\parens*{1+\sum_{\run=\nRuns_{0}}^{\nRuns}\alpha_{i}/\alpha_{0}}
	\notag\\
	&\leq \frac{2\alpha}{\alpha_{0}}
		+ 2
		+ 2\log\parens*{1 + \sum_{\run=1}^{\nRuns}\alpha_{i}/\alpha_{0}}
\end{align}
where we used the fact that $\sum_{j=1}^{\nRuns_{0}-2}\alpha_{j}\leq \alpha$ as well as for all $\run \geq \nRuns_{0}, \sum_{j=1}^{\run-1}\alpha_{j}\geq \alpha$ (both follow from the definition of $\nRuns_{0}$) and \cref{lem:logarithmic-1}.
\end{proof}

\section{Fisher markets: A case study}
\label{app:Fisher}

\subsection{The Fisher market model}

In this appendix, we illustrate the convergence properties of \method in a Fisher equilibrium problem with linear utilities \textendash\ both stochastic and deterministic.
Following \cite{NRTV07}, a Fisher market consists of a set $\players = \{1,\dotsc,\nPlayers\}$ of $\nPlayers$ \emph{buyers} \textendash\ or \emph{players} \textendash\ that seek to share a set $\pures = \{1,\dotsc,\nPures\}$ of $\nPures$ perfectly divisible goods (ad space, CPU/GPU runtime, bandwidth, etc.).
The allocation mechanism for these goods follows a proportionally fair price-setting rule that is sometimes referred to as a \emph{Kelly auction} \citep{KMT98}:
each player $\play = 1,\dotsc,\nPlayers$ bids $\point_{\play\pure}$ per unit of the $\pure$-th good, up the player's individual budget;
for the sake of simplicity, we assume that this budget is equal to $1$ for all players, so $\sum_{\pure=1}^{\nPures} \point_{\play\pure} \leq 1$ for all $\play=1,\dotsc,\nPlayers$.
The price of the $\pure$-th good is then set to be the sum of the players' bids, \ie $\price_{\pure} = \sum_{\play\in\players} \point_{\play\pure}$;
then, each player gets a prorated fraction of each good, namely $\weight_{\play\pure} = \point_{\play\pure}/\price_{\pure}$.

Now, if the marginal utility of the $\play$-th player per unit of the $\pure$-th good is $\param_{\play\pure}$, the agent's total utility will be
\begin{equation}
\pay_{\play}(\point_{\play};\point_{-\play})
	= \sum_{\pure\in\pures} \param_{\play\pure} \weight_{\play\pure}
	= \sum_{\pure\in\pures} \frac{\param_{\play\pure} \point_{\play\pure}}{\sum_{\playalt\in\players} \point_{\playalt\pure}},
\end{equation}
where $\point_{\play} = (\point_{\play\pure})_{\pure\in\pures}$ denotes the bid profile of the $\play$-th player, and we use the shorthand $(\point_{\play};\point_{-\play}) = (\point_{1},\dotsc,\point_{\play},\dotsc,\point_{\nPlayers})$.
A \emph{Fisher equilibrium} is then reached when the players' prices bids follow a profile $\eq = (\eq_{1},\dotsc,\eq_{\nPlayers})$ such that
\begin{equation}
\label{eq:Nash}
\tag{Eq}
\pay_{\play}(\eq_{\play};\eq_{-\play})
	\geq \pay_{\play}(\point_{\play};\eq_{-\play})
\end{equation}
for all $\play\in\players$ and all $\point_{\play} = (\point_{\play\pure})_{\pure\in\pures}$ such that $\point_{\play\pure} \geq 0$ and $\sum_{\pure\in\pures} \point_{\play\pure} = 1$.%
\footnote{It is trivial to see that, in this market problem, all users would saturate their budget constraints at equilibrium, \ie $\sum_{\pure\in\pures}\point_{\play\pure} = 1$ for all $\play \in \players$.}

As was observed by \citet{Shm09}, the equilibrium problem \eqref{eq:Nash} can be rewritten equivalently as
\begin{equation}
\label{eq:opt}
\tag{Opt}
\begin{aligned}
\textrm{minimize}
	&\quad
	\sobj(\point;\param)
		\equiv \sum_{\pure\in\pures} \price_{\pure} \log\price_{\pure}
		- \sum_{\play\in\players} \sum_{\pure\in\pures} \point_{\play\pure} \log\param_{\play\pure}
	\\
\textrm{subject to}
	&\quad
	\price_{\pure} = \sum_{\play\in\players} \point_{\play\pure},\,
	\sum_{\pure\in\pures} \point_{\play\pure} = 1,\;
	\textrm{and}\;
	\point_{\play\pure} \geq 0
	\;
	\text{for all $\pure\in\pures$, $\play\in\players$},
\end{aligned}
\end{equation}
with the standard continuity convention $0\log0 = 0$.
In the above, the agents' marginal utilities are implicitly assumed fixed throughout the duration of the game.
On the other hand, if these utilities fluctuate stochastically over time, the corresponding reformulation instead involves the \emph{mean} objective
\begin{equation}
\label{eq:obj}
\obj(\point)
	= \exof{\sobj(\point;\sample)}.
\end{equation}
Because of the logarithmic terms involved, $\sobj$ (and, a fortiori, $\obj$) cannot be Lipschitz continuous or smooth in the standard sense.
However, as was shown by \citet{BDX11}, the problem satisfies \eqref{eq:RS} over $\points = \setdef{\point\in\R_{+}^{\nPlayers\nPures}}{\sum_{\pure\in\pures}\point_{\play\pure} = 1}$ relative to the negative entropy function $\hreg(\point) = \sum_{\play\pure} \point_{\play\pure}\log\point_{\play\pure}$.
As a result, \acl{MD} methods based on this Bregman function are natural candidates for solving \eqref{eq:obj}.


\begin{figure}[t]
\centering
\begin{subfigure}[b]{.49\linewidth}
\includegraphics[width=\textwidth]{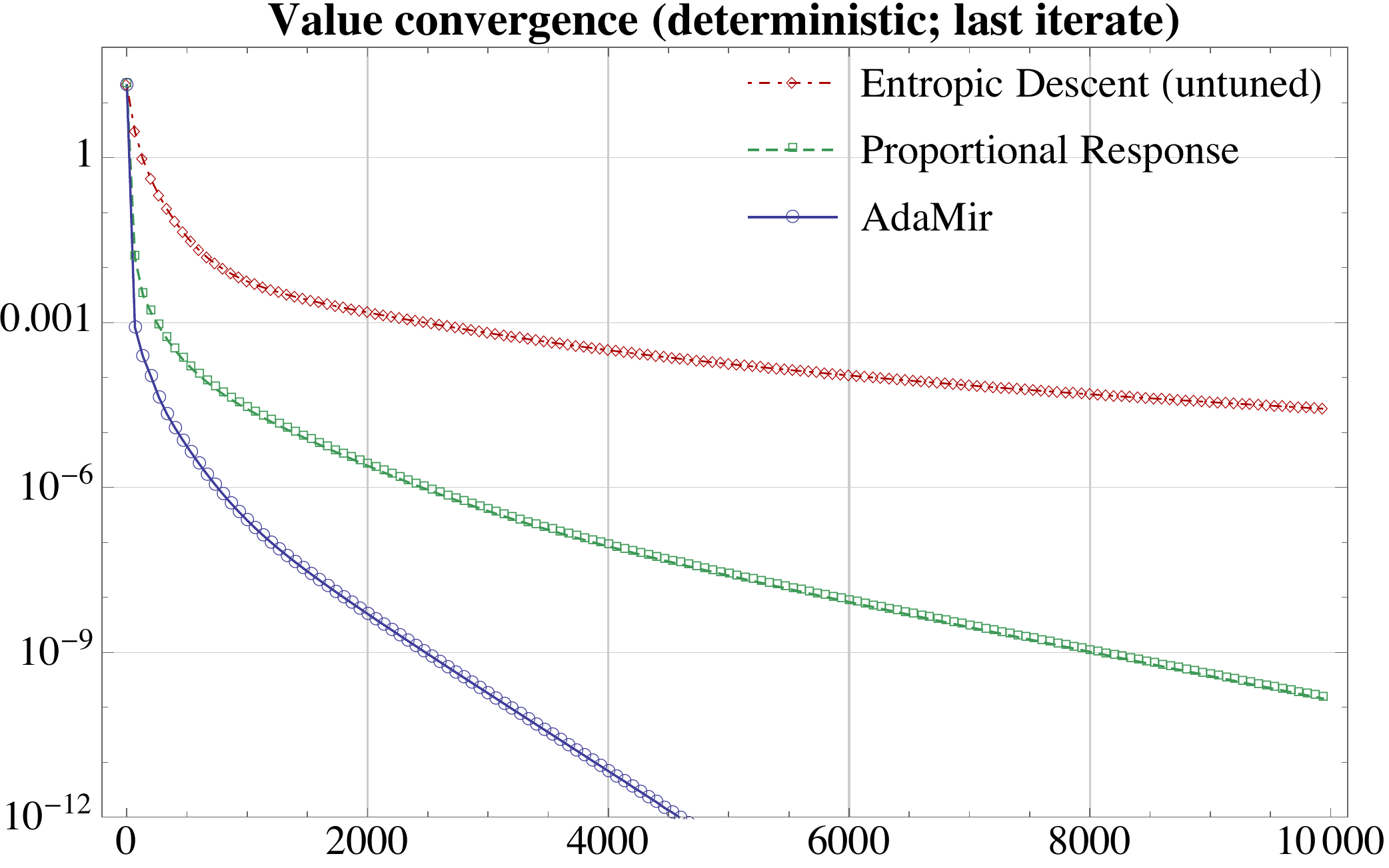}
\caption{Last-iterate convergence}
\label{fig:det-last}
\end{subfigure}
\hfill
\begin{subfigure}[b]{.475\linewidth}
\includegraphics[width=\textwidth]{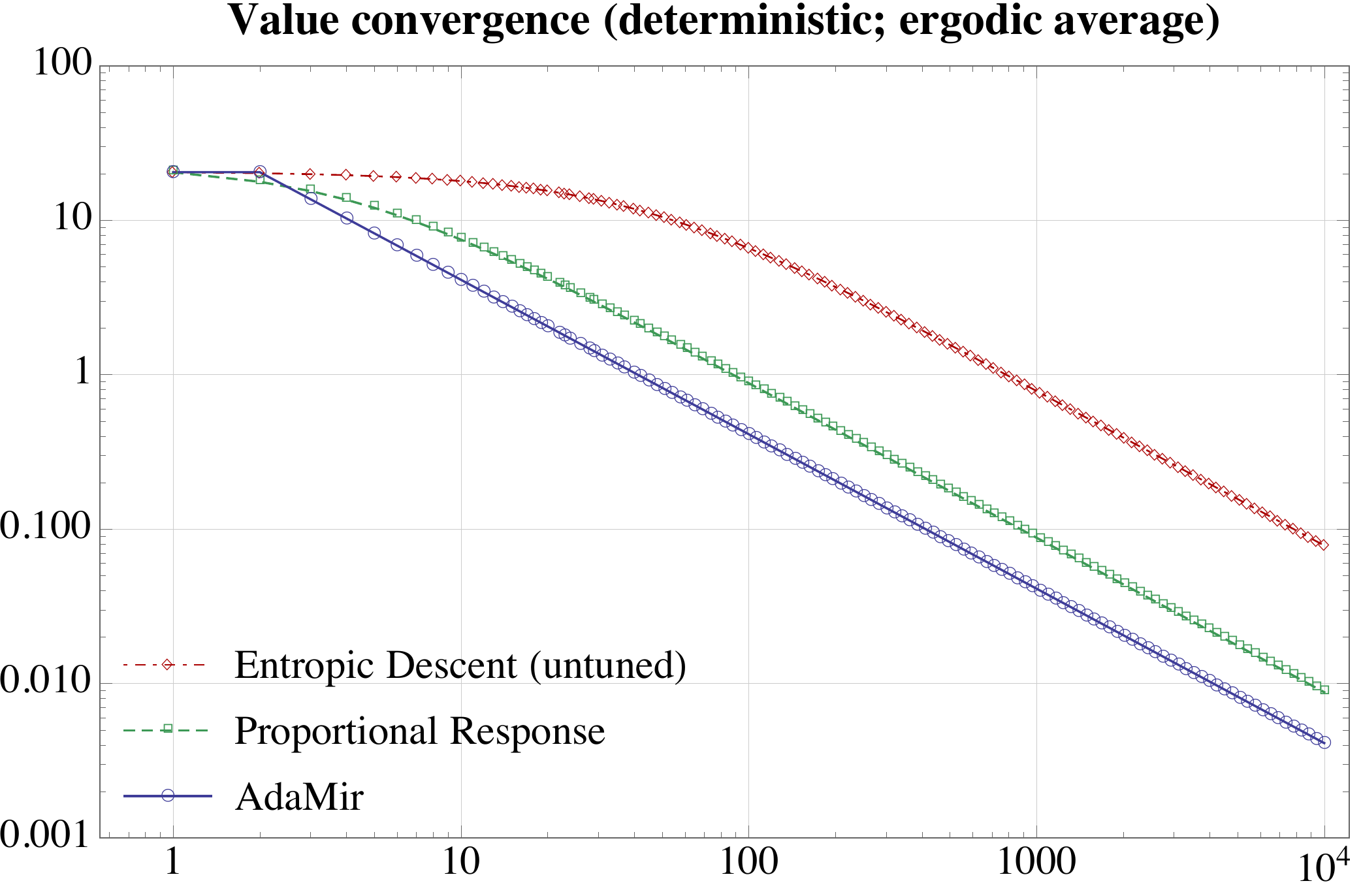}
\caption{Ergodic convergence}
\label{fig:det-erg}
\end{subfigure}
\caption{The convergence speed of \eqref{eq:EGD}, \eqref{eq:PR} and \method in a stationary Fisher market.}
\label{fig:det}
\end{figure}


In more detail, following standard arguments \citep{BecTeb03}, the general \acl{MD} template \eqref{eq:MD} relative to $\hreg$ can be written as
\begin{equation}
\new[\point_{\play\pure}]
	= \frac
		{\point_{\play\pure} \exp(-\step\gvec_{\play\pure})}
		{\sum_{\purealt\in\pures} \point_{\play\purealt} \exp(-\step\gvec_{\play\purealt})}
\end{equation}
where the (stochastic) gradient vector $\gvec \equiv \gvec(\point;\param)$ is given in components by
\begin{equation}
\gvec_{\play\pure}
	= 1 + \log \price_{\pure} - \log\param_{\play\pure}.
\end{equation}
Explicitly, this leads to the \acl{EGD} algorithm
\begin{equation}
\label{eq:EGD}
\tag{EGD}
\state_{\play\pure,\run+1}
	= \frac
		{\state_{\play\pure,\run} (\param_{\play\pure}/\price_{\pure})^{\step_{\run}}}
		{\sum_{\purealt\in\pures} \state_{\play\purealt,\run}(\param_{\play\purealt}/\price_{\purealt})^{\step_{\run}}}
\end{equation}
In particular, as a special case, the choice $\step = 1$ gives the \acdef{PR} algorithm of \citet{WZ07}, namely
\begin{equation}
\label{eq:PR}
\tag{PR}
\state_{\play\pure,\run+1}
	= \frac{\param_{\play\pure}\weight_{\play\pure,\run}}{\sum_{\purealt\in\pures} \param_{\play\purealt} \weight_{\play\purealt,\run}},
\end{equation}
where $\weight_{\play\pure,\run} = \state_{\play\pure,\run} \big/ \sum_{\playalt\in\players} \state_{\playalt\pure,\run}$.
As far as we aware, the \ac{PR} algorithm is considered to be the most efficient method for solving \emph{deterministic} Fisher equilibrium problems \citep{BDX11}.

\subsection{Experimental validation and methodology}

For validation purposes, we ran a series of numerical experiments on a synthetic Fisher market model with $\nPlayers = 50$ players sharing $\nPures = 5$ goods, and utilities drawn uniformly at random from the interval $[2,8]$.
For stationary markets, the players' marginal utilities were drawn at the outset of the game and were kept fixed throughout;
for stochastic models, the parameters were redrawn at each stage around the mean value of the stationary model (for consistency of comparisons).
All experiments were run on a MacBook Pro with a 6-Core Intel i7 CPU clocking in at 2.6GHZ and 16 GB of DDR4 RAM at 2667 MHz.
The Mathematica notebook used to generate the raw data and run the algorithms is included as part of the supplement (but not the entire sequence of random seed used in the stochastic case, as this would exceed the OpenReview upload limit).


\begin{figure}[t]
\centering
\begin{subfigure}[b]{.49\linewidth}
\includegraphics[width=\textwidth]{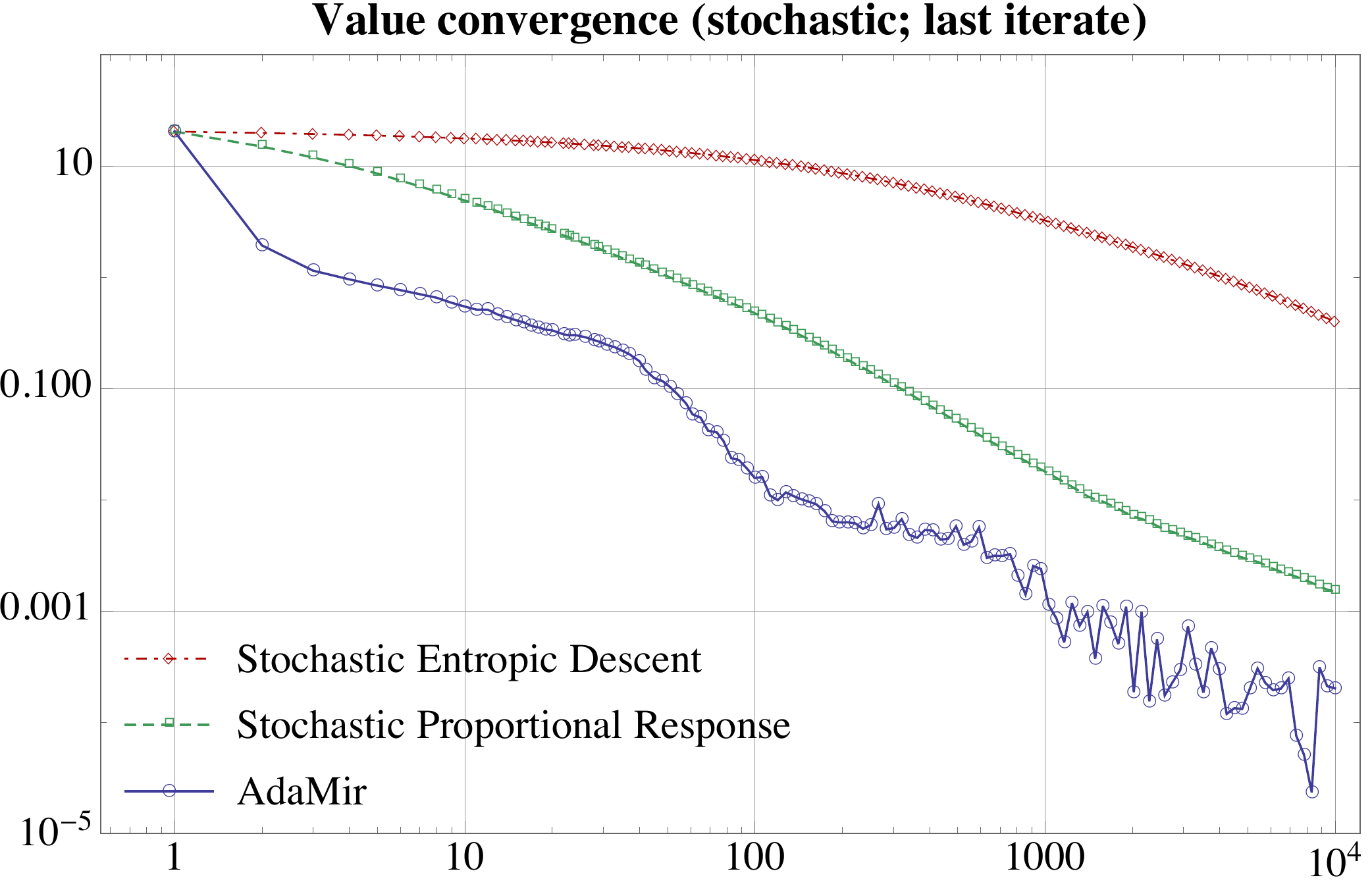}
\caption{Last-iterate convergence}
\label{fig:sample-last}
\end{subfigure}
\hfill
\begin{subfigure}[b]{.484\linewidth}
\includegraphics[width=\textwidth]{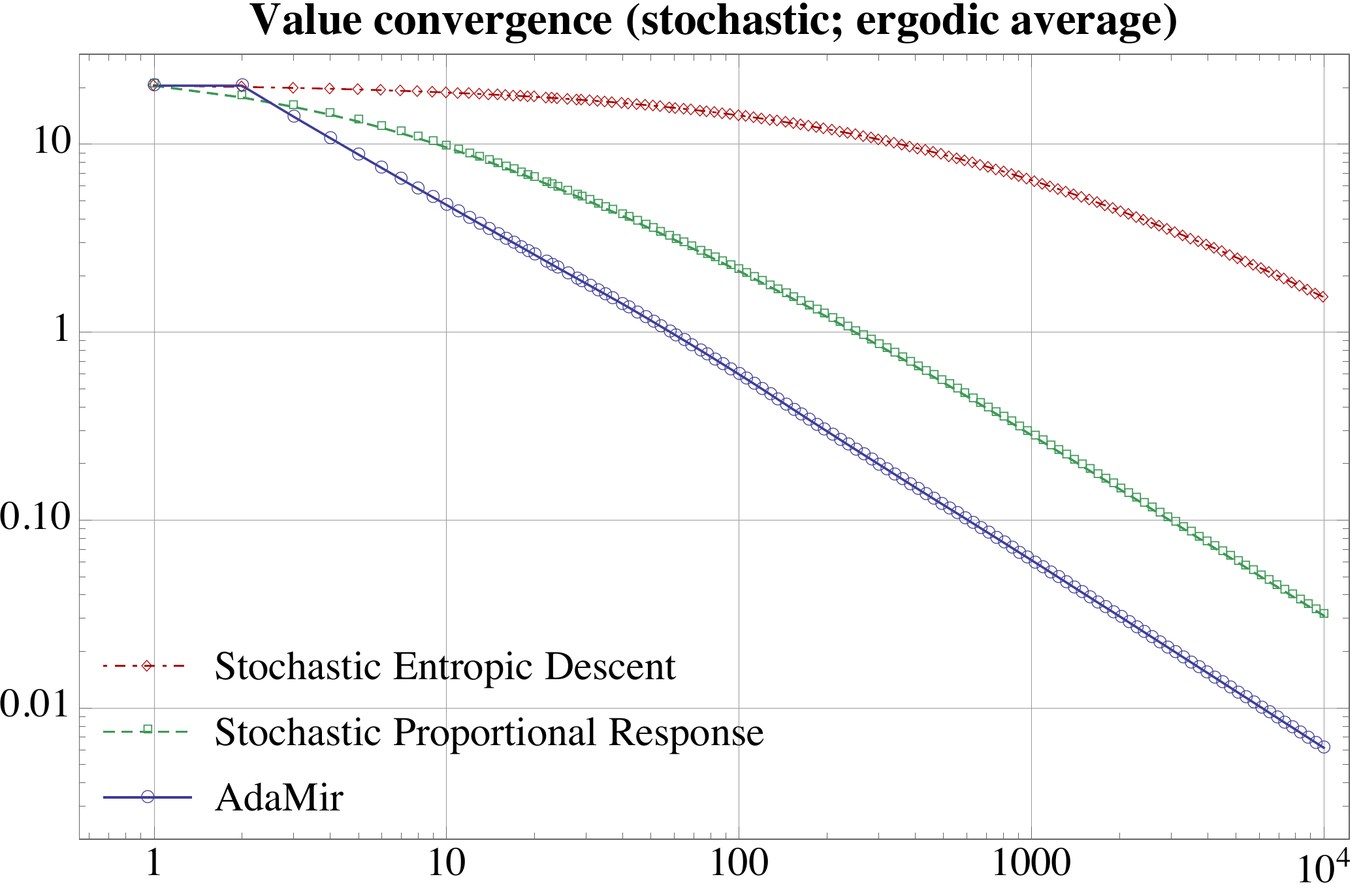}
\caption{Ergodic convergence}
\label{fig:sample-erg}
\end{subfigure}
\caption{The convergence speed of \eqref{eq:EGD}, \eqref{eq:PR} and \method in a stochastic Fisher market, with marginal utilities drawn \acs{iid} at each epoch.}
\label{fig:sample}
\end{figure}



\begin{figure}[tbp]
\centering
\begin{subfigure}[b]{.49\linewidth}
\includegraphics[height=\textwidth]{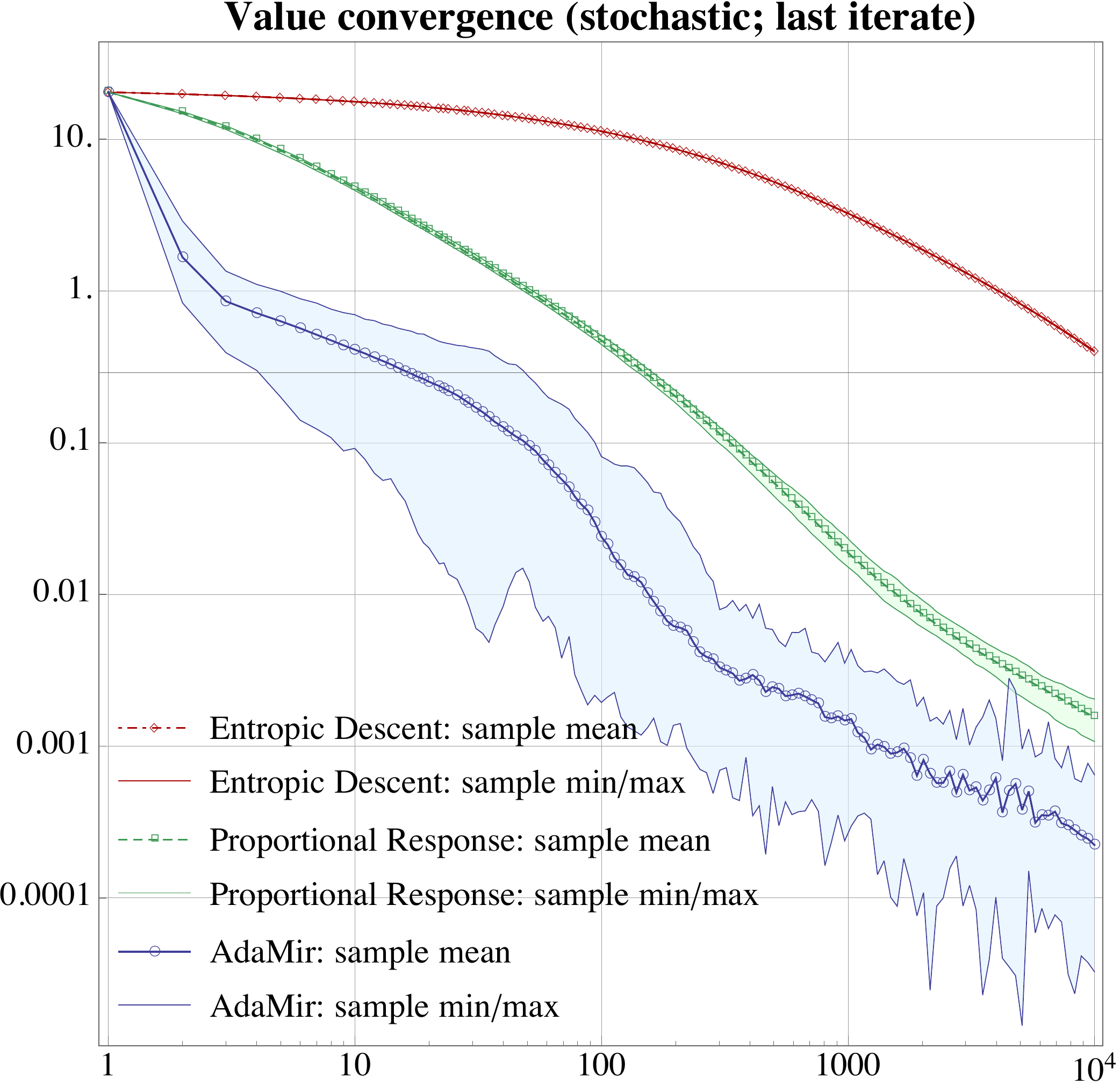}
\caption{Last-iterate convergence}
\label{fig:stoch-last}
\end{subfigure}
\hfill
\begin{subfigure}[b]{.49\linewidth}
\includegraphics[height=\textwidth]{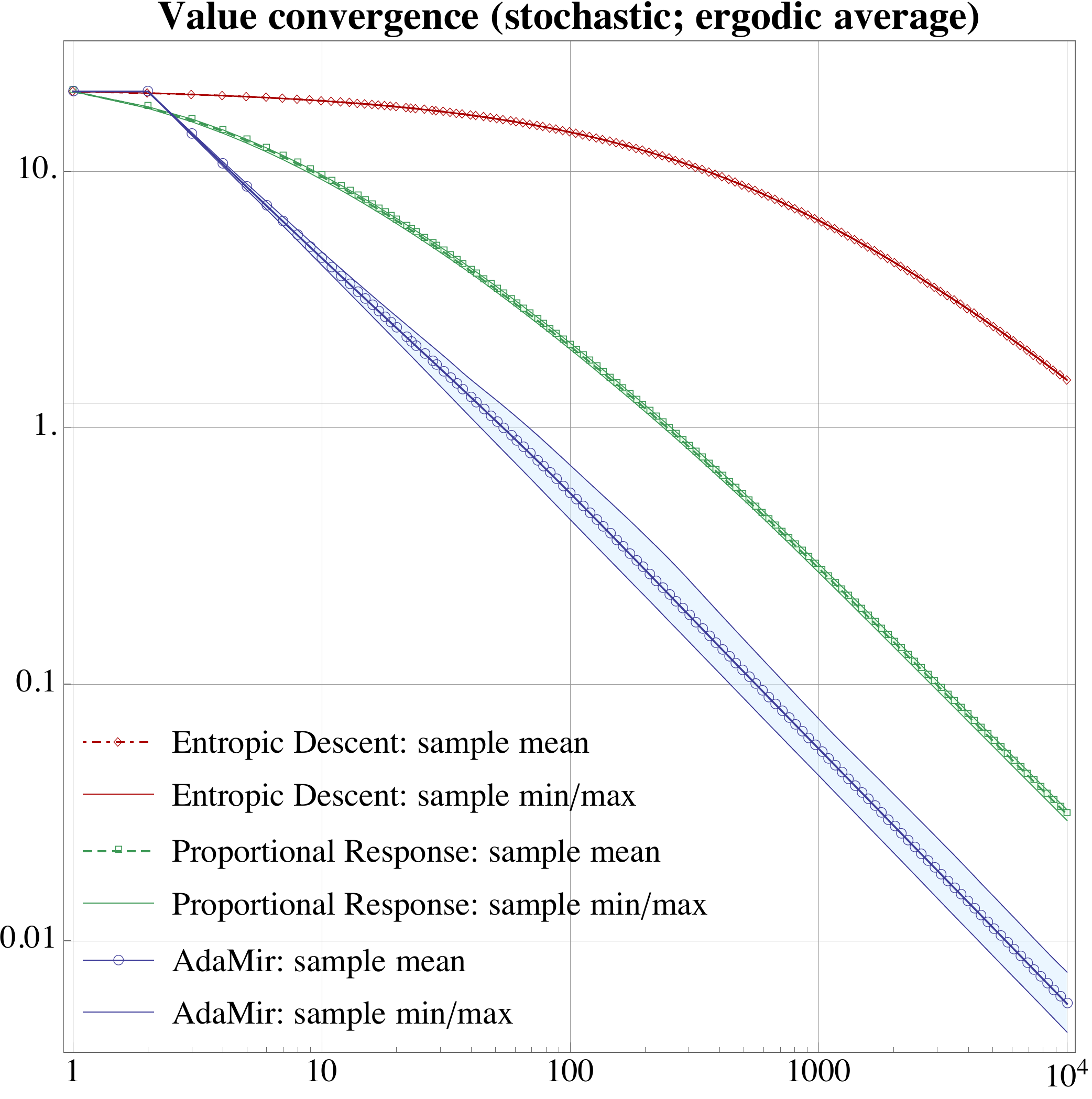}
\caption{Ergodic convergence}
\label{fig:stoch-erg}
\end{subfigure}
\caption{Statistics for the convergence speed of \eqref{eq:EGD}, \eqref{eq:PR} and \method in a stochastic Fisher market, with marginal utilities drawn \acs{iid} at each epoch.
The marked lines are the observed means from $S=50$ realizations, whereas the shaded areas represent a 95\% confidence interval.}
\label{fig:stoch}
\end{figure}


In each regime, we tested three algorithms, all initialized at the barycenter of $\points$:
\begin{enumerate*}
[\itshape a\upshape)]
\item
an untuned version of \eqref{eq:EGD};
\item
the \acl{PR} algorithm \eqref{eq:PR};
and
\item
\method.
\end{enumerate*}
For stationary markets, we ran the untuned version of \eqref{eq:EGD} with a step-size of $\step=.1$;
\eqref{eq:PR} was ran ``as is'', and \method was run with $\res_{\prestart}$ determined by drawing a second initial condition from $\points$.
In the stochastic case, following the theory of \citet{Lu19} and \citet{ABM20}, the updates of \eqref{eq:EGD} and \eqref{eq:PR} were modulated by a $\sqrt{\run}$ factor to maintain convergence;
by contrast, \method was run unchanged to test its adaptivity properties.

The results are reported in \cref{fig:det,fig:sample,fig:stoch}.
For completeness, we plot the evolution of each method in terms of values of $\obj$, both for the ``last iterate'' $\state_{\run}$ and the ``ergodic average'' $\bar\state_{\run}$.
The results for the deterministic case are presented in \cref{fig:det}.
For stochastic market models, we present a sample realization in \cref{fig:sample}, and a statistical study over $S=50$ sample realizations in \cref{fig:stoch}.
In all cases, \method outperforms both \eqref{eq:EGD} and \eqref{eq:PR}, in terms of both last-iterate and time-average guarantees.

An interesting observation is that each method's last iterate exhibits faster convergence than its time-average, and the convergence speed of the methods' time-averaged trajectories is faster than our worst-case predictions.
This is due to the specific properties of the Fisher market model under consideration:
more often than not, players tend to allocate all of their budget to a single good, so almost all of the problem's inequality constraints are saturated at equilibrium.
Geometrically, this means that the problem's solution lies in a low-dimensional face of $\points$, which is identified at a very fast rate, hence the observed accelerated rate of convergence.
However, this is a specificity of the market model under consideration and should not be extrapolated to other convex problems \textendash\ or other market equilibrium models to boot.

\section*{Acknowledgments}
{\small
This research was partially supported by
the French National Research Agency (ANR) in the framework of
the ``Investissements d'avenir'' program (ANR-15-IDEX-02),
the LabEx PERSYVAL (ANR-11-LABX-0025-01),
MIAI@Grenoble Alpes (ANR-19-P3IA-0003),
and the grants ORACLESS (ANR-16-CE33-0004) and ALIAS (ANR-19-CE48-0018-01).}

\bibliographystyle{plainnat}
\bibliography{IEEEabrv,bibtex/Bibliography-PM}

\end{document}